\documentclass[12pt]{amsart}
\usepackage{a4wide}
\usepackage{graphicx}
\usepackage{color}
\usepackage{amsmath}
\usepackage{mathabx} 
\usepackage{mathbbol} 
\allowdisplaybreaks

\let\pa\partial
\let\na\nabla
\let\eps\varepsilon
\newcommand{\N}{{\mathbb N}}
\newcommand{\R}{{\mathbb R}}
\newcommand{\diver}{\operatorname{div}}

\newcommand{\blue}[1]{{\color{blue}{#1}}}  

\newcommand{\E}{\mathcal E}
\newcommand{\T}{\mathcal T}
\newcommand{\aver}[1]{\{#1\}}
\newcommand{\jump}[1]{\ldbrack#1\rdbrack}

\newtheorem{theorem}{Theorem}
\newtheorem{lemma}[theorem]{Lemma}
\newtheorem{proposition}[theorem]{Proposition}
\newtheorem{remark}[theorem]{Remark}

\begin{document}

\title[A structure-preserving DG scheme]{A structure-preserving discontinuous Galerkin
scheme for the Fisher-KPP equation}

\author[F. Bonizzoni]{Francesca Bonizzoni}
\address{Faculty of Mathematics, University of Vienna, Oskar-Morgenstein-Platz 1,
1090 Wien, Austria}
\email{francesca.bonizzoni@univie.ac.at}

\author[M. Braukhoff]{Marcel Braukhoff}
\address{Institute for Analysis and Scientific Computing, Vienna University of
	Technology, Wiedner Hauptstra\ss e 8--10, 1040 Wien, Austria}
\email{marcel.braukhoff@tuwien.ac.at}

\author[A. J\"ungel]{Ansgar J\"ungel}
\address{Institute for Analysis and Scientific Computing, Vienna University of
	Technology, Wiedner Hauptstra\ss e 8--10, 1040 Wien, Austria}
\email{juengel@tuwien.ac.at}

\author[I. Perugia]{Ilaria Perugia}
\address{Faculty of Mathematics, University of Vienna, Oskar-Morgenstein-Platz 1,
1090 Wien, Austria}
\email{ilaria.perugia@univie.ac.at}

\date{\today}

\thanks{The authors acknowledge partial support from the Austrian Science Fund (FWF), 
grant F65. The first author has been supported by the FWF Firnberg-Program, grant 
T998. The second and third authors acknowledge support from the FWF, grants
P30000 and W1245. The forth author acknowledges support from the FWF, grant P29197-N32.
The authors thank Cl\'ement Canc\`es for pointing out how to prove
Proposition \ref{prop.p1}.}

\begin{abstract}
An implicit Euler discontinuous Galerkin scheme for the
Fisher-Kol\-mo\-go\-rov-Petrovsky-Piscounov (Fisher-KPP) equation for
population densities with no-flux boundary conditions is suggested and
analyzed. Using an exponential variable transformation, the numerical scheme
automatically preserves the positivity of the discrete solution. 
A discrete entropy inequality
is derived, and the exponential time decay of the discrete density to the
stable steady state in the $L^1$ norm is proved if the initial entropy
is smaller than the measure of the domain. The discrete solution is
proved to converge in the $L^2$ norm to the unique strong solution to the
time-discrete Fisher-KPP equation as the mesh size tends to zero.
Numerical experiments in one space dimension illustrate the theoretical results. 
\end{abstract}



\maketitle


\section{Introduction}

The preservation of the structure of nonlinear diffusion equations 
on the discrete level is of paramount importance in applications. 
While there has been an enormous progress on structure-preserving schemes
for ordinary differential equations (see, e.g., \cite{HLW06}),
the development of structure-preserving numerical techniques for nonlinear diffusion
equations is still an
ongoing quest, in particular for higher-order methods. In this paper, we
analyze a toy problem, the Fisher-Kolmogorov-Petrovsky-Piscounov (Fisher-KPP)
equation with no-flux boundary conditions, to devise an implicit Euler
discontinuous Galerkin scheme which preserves the positivity of the solution,
the entropy structure, and the exponential equilibration on the discrete level.
In a future work, we aim to extend the scheme to diffusion systems.

The Fisher-KPP equation \cite{Fis37} is the reaction-diffusion equation
\begin{align}
  & \pa_t u = D\Delta u + u(1-u)\quad\mbox{in }\Omega,\ t>0, \label{1.eq} \\
	& \na u\cdot n = 0\quad\mbox{on }\pa\Omega, \quad u(0) = u_0\quad\mbox{in }
	\Omega, \label{1.bic}
\end{align}
where $D>0$ is the diffusion coefficient, $\Omega\subset\R^d$ a bounded domain,
and $n$ the exterior unit normal vector on the boundary $\pa\Omega$. The
variable $u(x,t)$ models a population density or chemical concentration, 
influenced by diffusion
and logistic growth. The Fisher-KPP equation admits traveling-wave solutions
$u(x,t)=\phi(x-ct)$, which switch between the unstable steady state $u^*=0$ and
the stable steady state $u^*=1$. By the maxiumum principle, the density stays
nonnegative if it does so initially, and it satisfies the entropy inequality
\begin{equation}\label{1.epi}
  \frac{d}{dt}\int_\Omega u(\log u-1)dx + D\int_\Omega\frac{|\na u|^2}{u}dx
	= -\int_\Omega u(u-1)\log u dx \le 0.
\end{equation}

If there are no reaction terms, we have conservation of the total mass, and
the logarithmic Sobolev inequality implies
the exponential decay of the (mathematical) entropy 
$S(t)=\int_\Omega (u(t)(\log u(t)-1)+1)dx$ (see, e.g., \cite[Chapter 2]{Jue16}).
When reaction terms are present, the situation is more delicate, since there
are two steady states, $u^*=0$ and $u^*=1$. 
If the initial entropy $S(0)$ is smaller than
the measure of $\Omega$, then $u(t)$ converges exponentially fast to
$u^*=1$ in the $L^1(\Omega)$ norm. Our objective is to preserve the aforementioned
properties on the discrete level.

It is well known that the preservation of the positivity or nonnegativity of
discrete solutions for \eqref{1.eq} may fail in standard (finite-element) schemes,
in particular when the solution vanishes in some region; see Section
\ref{sec.num} for an example. 
Our key idea to preserve the positivity is to employ the
exponential transformation $u=e^\lambda$. 
Such a transformation or a variant is used, for instance, in the
Il'in scheme \cite{Ili69} and in the existence analysis of 
drift-diffusion equations \cite{Gaj85}. Moreover, it allows for the preservation
of $L^\infty(\Omega)$ bounds in volume-filling cross-diffusion systems
\cite{BDPS10,Jue16}. The implicit Euler scheme for \eqref{1.eq}-\eqref{1.bic}
in the exponential variable then reads as
\begin{align}
  \frac{1}{\triangle t}\big(e^{\lambda^k} - e^{\lambda^{k-1}}\big)
	&= \diver(e^{\lambda^k}\na e^{\lambda^k}) + e^{\lambda^k}(1-e^{\lambda^k})
	\quad\mbox{in }\Omega, \label{sc.eq} \\
	\na\lambda^k\cdot n &= 0 \quad\mbox{on }\pa\Omega,
  \label{sc.bc}
\end{align}
where here and in the following, we set $D=1$ for simplicity and we choose
$0<\triangle t<1$. At first glance, one
may think that this formulation unnecessarily complicates the problem,
but we will show that it enjoys some useful properties.

We propose a discontinuous Galerkin (DG) discretization
for problem \eqref{sc.eq}-\eqref{sc.bc}
with variable $\lambda_h^k$, where $h>0$ is the maximal diameter of the mesh elements.
The nonlinear diffusion term is discretized by an interior penalty DG method. 
By construction, the discrete densities $\exp(\lambda_h^k)$ are positive, and
the scheme also preserves the entropy structure and large-time asymptotics.
Our main results can be sketched as follows:
\begin{itemize}
\item Existence of a solution $\lambda_h^k$ to the implicit Euler DG scheme 
\eqref{dg}, given a function $\lambda_h^{k-1}$ (Proposition \ref{prop.ex}). 
This result is based on the Leray-Schauder fixed-point theorem
and a coercivity estimate.
\item Discrete entropy inequality (Lemma \ref{lem.epi}). The inequality follows
from scheme \eqref{dg} using the test function $\lambda_h^k$ and the
convexity of $u\mapsto u(\log u-1)+1$.
\item Exponential decay of the discrete entropy 
$$
  S_h^k := \int_\Omega\big(e^{\lambda_h^k}(e^{\lambda_h^k}-1)+1\big)dx
	\le S_h^0 e^{-\kappa k\triangle t}
$$
(Proposition \ref{prop.decay}) and of the $L^1$ norm of $e^{\lambda_h^k}-1$
(Theorem \ref{thm.decay}). The result holds if $S_h^0<|\Omega|$. This condition implies
a positive lower bound for the total mass $\int_\Omega e^{\lambda_h^k}dx$, which
is needed to guarantee that the discrete solution converges to the stable
steady state $u^*=1$ and not to the steady state $u^*=0$. 
The case $S_h^0\ge|\Omega|$ is discussed in Remark \ref{rem.S0}.
\item Convergence of the scheme (Theorem \ref{thm.conv}): There exists
a unique strong solution $u^k\in H^2_n(\Omega)$ to the implicit Euler
discretization associated to \eqref{1.eq}-\eqref{1.bic} such that
$$
  e^{\lambda_h^k}\to u^k\quad\mbox{strongly in }L^2(\Omega)\mbox{ as }
	h\to 0.
$$
The result is based on a compactness property, which is a consequence of
the gradient estimate from the entropy inequality and a coercivity
estimate. This yields a very weak semi-discrete solution,
which turns out to be a strong solution thanks to a duality argument.
\end{itemize}

Let us put our results into context and review the state of the art of 
stucture preservation in DG methods.
The DG scheme was introduced in the early 1970s for first-order hyperbolic problems in 
\cite{LeRa74,ReHi73}. The development of discontinuous finite-element schemes for 
second-order elliptic problems can be traced back to \cite{Nit71} with similar
approaches in, for instance, \cite{Arn82,Bak77,PeWh78,Whe78}; see also \cite{ABCM02}.

The design of structure-preserving DG methods is a rather recent topic.
Positivity-preserving DG schemes for parabolic equations were developed
in, e.g., \cite{CNP12,GuYa15,LiWa16,SCS18,ZZJLAY16}. 
The positivity preservation is ensured
by using a special slope limiter (as in \cite{CNP12,GuYa15}), together with a  
strong stability preserving Runge-Kutta time discretization 
(as in \cite{SCS18,ZZJLAY16}), while in \cite{LiWa16}, the positivity of the
discrete solution is enforced through a reconstruction algorithm, based on 
positive cell averages. As far as we know, the use of an exponential transformation
to ensure the positivity of the discrete solutions within a DG scheme is new.
Positivity-preserving schemes for the Fisher-KPP equation were already studied in the
literature, but only for finite-difference approximations \cite{HSS17,MJP12}, 
without a convergence analysis, and for continuous finite-element 
discretizations \cite{YaJi17}.

Other important properties are entropy stability (the entropy is bounded for all 
times) and entropy monotonicity (the entropy is nonincreasing). Entropy-stable
DG schemes for the compressible Euler and Navier-Stokes equation were studied 
in \cite{GWK16,PCFN16}, while a discrete version of the entropy inequality
(and hence entropy monotonicity) was proved in \cite{SCS18} for Fokker-Planck-type
equations and aggregation models. We are not aware of results in the
literature regarding the preservation of the entropy structure of the Fisher-KPP
equation on the discrete level.

The paper is organized as follows. We state our notation and some auxiliary
results related to the DG method in Section \ref{sec.aux}. The DG scheme is
introduced and studied in Section \ref{sec.anal}: 
The existence of a solution to the DG scheme,
the discrete entropy inequality, and the exponential decay of the
entropy are proved. The convergence of the numerical scheme is
proved in Section~\ref{sec.conv}. Finally, Section \ref{sec.num} is devoted to some
numerical experiments in one space dimension.


\section{Notation and auxiliary results}\label{sec.aux}

We start with some notation. Let $\T_h=\{K_i:i=1,\ldots,N_h\}$ be a family of
simplicial partitions of the bounded domain $\Omega\subset\R^d$
for $d=1,2,3$. The mesh parameter $h$ is defined by $h=\max_{K\in\T_h}h_K$,
where $h_K=\operatorname{diam}(K)$. The elements may be tetrahedra in three
space dimensions, triangles in two dimensions, and intervals in one dimension.
In two and three dimensions, we suppose that $\T_h$ is shape regular
(see, e.g., \cite[Section 2.1]{PPOR00}) and, for simplicity, without hanging nodes.
Our analysis actually extends also to $k$-irregular meshes
\blue{\cite{HSS02}.}
We denote by $\E_h$ the set of interior faces or edges of the elements in $\T_h$.

On the partition $\T_h$, we define the broken Sobolev space 
$$
  H^s(\Omega,\T_h) = \big\{\xi\in L^2(\Omega):\xi|_K\in H^s(K)
	\mbox{ for all }K\in\T_h\big\}, \quad s>0.
$$
The traces of functions in $H^1(\Omega,\T_h)$ belong to the space
$T(\Gamma_h)=\prod_{K\in\T_h}L^2(\pa K)$, where $\Gamma_h$ is the union of all 
boundaries $\pa K$ for all $K\in\T_h$. The functions in $T(\Gamma_h)$ 
are single-valued on $\pa\Omega$ and double-valued on $\Gamma_h\setminus\pa\Omega$.

Let $q$ be a piecewise smooth function and $q$ be a piecewise smooth vector field on
$\T_h$. We write $K_-$ and $K_+$ for the two elements sharing 
the face $f$, i.e.\ $f=\pa K_-\cap\pa K_+$, and $n_\pm$ for the unit normal
vector pointing to the exterior of $K_\pm$. Furthermore, we set
$q_\pm = q|_{K_\pm}$ and $\phi_\pm = \phi|_{K_\pm}$. Then we define
\begin{align*}
  \mbox{averages:}\quad & \aver{q} = \frac12(q_- + q_+), \quad 
	\aver{\phi} = \frac12(\phi_- + \phi_+), \\
	\mbox{jumps:}\quad & \jump{q} = q_-n_- + q_+n_+, \quad 
	\jump{\phi} = \phi_-\cdot n_- + \phi_+\cdot n_+.
\end{align*}
Note that the jump of a scalar function is a vector which is normal
to $f$, and the jump of a vector-valued function is a scalar.

The mesh size function ${\tt h}\in L^\infty(\Gamma_h)$ is defined by
$$
  {\tt h}(x) = \min\{h_{K_-},h_{K_+}\} \quad\mbox{for }x\in\pa K_-\cap\pa K_+.
$$
Furthermore, we introduce the finite-element space of degree $p\in\N$ associated to
the partition $\T_h$:
$$
  V_h = \big\{v\in L^2(\Omega): v|_K\in P_p(K)\mbox{ for all }K\in\T_h\big\},
$$
where $P_p(K)$ is the set of polynomials on $K$ with degree at most $p$,
and the space of test functions
$$
  H_n^2(\Omega) = \{\phi\in H^2(\Omega):\na\phi\cdot n=0\mbox{ on }\pa\Omega\}.
$$

Next, we recall some auxiliary results.

\begin{lemma}[Inverse trace inequality; Lemma 2.1 in \cite{RWG01}]\label{lem.inv}
Let $K\in\R^d$ ($d=2,3$) be an element with diameter $h_K$, let $f$ be an edge or
face of $K$, and let $n_f$ be a unit normal vector normal to $f$. Then for
all polynomials $\xi\in P_p(K)$ of degree $p$, there exists a constant $C_{\rm inv}>0$,
independent of $h_K$ and $p$, such that
\begin{align}
  \|\xi\|_{L^2(\pa K)} &\le C_{\rm inv}\frac{p}{\sqrt{h_K}}\|\xi\|_{L^2(K)}, 
	\label{inv.ineq} \\
	\|\na\xi\cdot n_f\|_{L^2(\pa K)} 
	&\le C_{\rm inv}\frac{p}{\sqrt{h_K}}\|\na\xi\|_{L^2(K)}. \nonumber
\end{align}
\end{lemma}

\begin{lemma}[Multiplicative trace inequality; Lemma A.2 in \cite{PPOR00}]\label{lem.mti}
Let $K$ be a shape-regular element. Then there exists a constant $C>0$
such that for all $\xi\in H^1(K)$,
$$
  \|\xi\|_{L^2(\pa K)}^2 \le C\|\xi\|_{L^2(K)}\bigg(\frac{1}{h_K}\|\xi\|_{L^2(K)}
	+ \|\na\xi\|_{L^2(K)}\bigg).
$$
\end{lemma}

\begin{lemma}[Discrete Poincar\'e-Wirtinger inequality; Theorem 4.1 in
\cite{BuOr09}]\label{lem.dpw}\ \
There exists a constant $C_{\rm PW}>0$ such that for all $\xi\in H^1(\Omega,\T_h)$,
$$
  \bigg\|\xi - \frac{1}{|\Omega|}\int_\Omega\xi dx\bigg\|_{L^2(\Omega)}
	\le C_{\rm PW}\bigg(\sum_{K\in \T_h}\|\na\xi\|_{L^2(K)}^2
	+ \sum_{f\in\E_h}\int_f\frac{p^2}{{\tt h}}|\jump{\xi}|^2 dx\bigg)^{1/2}.
$$
\end{lemma}

We also need a compactness result for functions $\xi\in H^1(\Omega,\T_h)$.
For this, we define the DG norm
\begin{equation}\label{dgnorm}
  \|\xi\|_{\rm DG} = \bigg(\|\xi\|_{L^2(\Omega)}^2 
	+ \sum_{K\in\T_h}\|\na\xi\|_{L^2(K)}^2
	+ \sum_{f\in\E_h}\int_f\frac{p^2}{{\tt h}}|\jump{\xi}|^2 dx\bigg)^{1/2}.
\end{equation}

\begin{lemma}[DG compact embedding; Lemma 8 in \cite{BuOr09}]\label{lem.comp}
Let $(\xi_h)\subset H^1(\Omega,\T_h)$ be a sequence such that $\|\xi_h\|_{\rm DG}\le C$
for all $h\in(0,1)$ and some $C>0$. Then there exists a subsequence $(h_i)$ with
$h_i\to 0$ as $i\to\infty$ and a function $\xi\in H^1(\Omega)$ such that
$$
  \xi_{h_i}\to\xi \quad\mbox{strongly in }L^q(\Omega)\mbox{ as }h_i\to 0,
$$
where $1\le q<q^*$ and $q^*=4$ for $d=3$, $q^*=\infty$ for $d=1,2$.
\end{lemma}


\section{Analysis of the DG scheme: existence and 
structure preservation}\label{sec.anal}

The DG discretization of the weak formulation of \eqref{sc.eq}-\eqref{sc.bc}
reads as follows. Let $\eps\ge 0$ and $\lambda_h^0\in V_h$. Given
$\lambda_h^{k-1}\in V_h$, we wish to find $\lambda_h^k\in V_h$ such that
for all $\phi_h\in V_h$,
\begin{equation}\label{dg}
  \int_\Omega \big(e^{\lambda_h^k}-e^{\lambda_h^{k-1}}\big)\phi_h dx
  + \triangle t B(\lambda_h^k;\lambda_h^k,\phi_h)
  + \eps\int_\Omega\lambda_h^k\phi_h dx
  = \triangle t\int_\Omega e^{\lambda_h^k}(1-e^{\lambda_h^k})\phi_h dx.
\end{equation}
The form $B:H^1(\Omega,\T_h)^3\to\R$ represents the interior penalty DG discretization
of the nonlinear diffusion term. It is linear in the second and third argument
and is defined by
\begin{align}
  B(u;v,w) &= \sum_{K\in\T_h} e^u\na v\cdot\na w dx
	- \sum_{f\in\E_h}\int_f\big(\aver{e^u\na v}\cdot\jump{w} 
	+ \aver{e^u\na w}\cdot\jump{v}\big)ds \nonumber \\
	&\phantom{xx}{}+ \sum_{f\in\E_h}\int_f\frac{p^2}{{\tt h}}\alpha(u)\jump{v}\cdot
	\jump{w}ds, \label{def.B}
\end{align}
where $\alpha(u)$ is a stabilization function, given by
\begin{equation}\label{def.alpha}
  \alpha(u) = \frac32 C_{\rm inv}^2\big(\max\{(e^u)_-,(e^u)_+\}\big)^2
	\max\big\{\exp(\|u\|_{L^\infty(K_-)}),\exp(\|u\|_{L^\infty(K_+)})\big\}.
\end{equation}
We recall that the constant $C_{\rm inv}$ is defined in Lemma \ref{lem.inv}.
The third term on the left-hand side of \eqref{dg} is a regularization term
(only) needed for the existence analysis to derive a uniform (but $\eps$-depending)
bound for the fixed-point argument. For linear elements $p=1$, we may allow for
$\eps=0$; see Appendix \ref{app}.


\subsection{Existence of a discrete solution}

We show that problem \eqref{dg} possesses a solution. First, we prove a
coercivity property for the form $B$.

\begin{lemma}[Coercivity of $B$]\label{lem.coerc}
The form $B$, defined in \eqref{def.B}, satisfies for all $v\in H^1(\Omega,\T_h)$,
$$
  B(v;v,v) \ge 2\sum_{K\in\T_h}\int_K|\na e^{v/2}|^2 dx
	+ 2C_{\rm inv}^2\sum_{f\in\E_h}\int_f\frac{p^2}{{\tt h}}|\jump{e^{v/2}}|^2 ds.
$$
\end{lemma}

\begin{proof}
Definition \eqref{def.B} gives for $v\in H^1(\Omega,\T_h)$:
\begin{equation}
  B(v;v,v) = \sum_{K\in\T_h}\int_K e^v|\na v|^2 dx - 2\sum_{f\in\E_h}\int_f
	\aver{e^v\na v}\cdot\jump{v}ds
	+ \sum_{f\in\E_h}\int_f\frac{p^2}{{\tt h}}\alpha(v)|\jump{v}|^2 ds \label{2.aux1}
\end{equation}
We estimate the second integral by using Young's inequality:
$$
  2\sum_{f\in\E_h}\int_f	\aver{e^v\na v}\cdot\jump{v}ds 
	\le \sum_{f\in\E_h}\bigg(\int_f\beta_f^2\aver{e^v\na v}^2 ds
	+ \int_f\frac{1}{\beta_f^2}|\jump{v}|^2 ds\bigg),
$$
where $\beta_f>0$ is a parameter which will be defined below. The first integral
on the right-hand side is estimated according to
\begin{align*}
  \sum_{f\in\E_h}\int_f\beta_f^2\aver{e^v\na v}^2 ds
	&= \frac14\sum_{f\in\E_h}\int_f\beta_f^2\big|(e^v\na v)_- + (e^v\na v)_+\big|^2 ds \\
	&\le \frac12\sum_{f\in\E_h}\int_f\beta_f^2\big(|(e^v\na v)_-|^2
	+ |(e^v\na v)_+|^2\big)ds \\
	&= \frac12\sum_{f\in\E_h}\int_f\beta_f^2\big(\max\{(e^v)_-,(e^v)_+\}\big)^2
	\big(|(\na v)_-|^2 + |(\na v)_+|^2\big) ds.
\end{align*}
To proceed, we set 
$$
  \beta_f := \frac{\min\{\gamma_{K_-},\gamma_{K_+}\}}{\max\{(e^v)_-,(e^v)_+\}},
	\quad\mbox{where }\gamma_{K}^2:=\frac{h_{K}}{C_{\rm inv}^2p^2}
	\exp(-\|v\|_{L^\infty(K)}).
$$
Taking into account the inverse trace inequality \eqref{inv.ineq}, we infer that
\begin{align*}
  \sum_{f\in\E_h}\int_f\beta_f^2\aver{e^v\na v}^2 ds
	&\le \frac12\sum_{K\in\T_h}\int_{\pa K}\gamma_K^2|\na v|^2 ds
	\le \frac12 C_{\rm inv}^2\sum_{K\in\T_h}\gamma_K^2\frac{p^2}{h_K}\int_K|\na v|^2 dx \\
	&\le \frac12 C_{\rm inv}^2 p^2\sum_{K\in\T_h}\frac{\gamma_K^2}{h_K}
	\exp(\|v\|_{L^\infty(K)})\int_K e^v|\na v|^2 dx \\
	&= \frac12\sum_{K\in\T_h}\int_K e^v|\na v|^2 dx.
\end{align*}
Consequently, we obtain
$$
  2\sum_{f\in\E_h}\int_f	\aver{e^v\na v}\cdot\jump{v}ds  
	\le \sum_{f\in\E_h}\int_f\frac{1}{\beta_f^2}|\jump{v}|^2 ds
	+ \frac12\sum_{K\in\T_h}\int_K e^v|\na v|^2 dx.
$$
Inserting this estimate into \eqref{2.aux1}, it follows that
$$
  B(v;v,v) \ge \frac12\sum_{K\in\T_h} e^v|\na v|^2 dx
	+ \sum_{f\in\E_h}\int_f\bigg(\frac{p^2}{{\tt h}}\alpha(v) - \frac{1}{\beta_f^2}\bigg)
	|\jump{v}|^2 ds.
$$
With the definitions of $\alpha(v)$ (see \eqref{def.alpha})
and $\beta_f$ as well as the property ${\tt h}\le h_{K_\pm}$, 
the difference in the bracket can be computed as
\begin{align*}
  \frac{p^2}{{\tt h}}\alpha(v) - \frac{1}{\beta_f^2}
	&\ge \frac32\frac{p^2}{{\tt h}}C_{\rm inv}^2\big(\max\{(e^v)_-,(e^v)_+\}\big)^2
	\max\big\{\exp(\|v\|_{L^\infty(K_-)}),\exp(\|v\|_{L^\infty(K_+)})\big\} \\
	&\phantom{xx}{}- \frac{C_{\rm inv}^2p^2(\max\{(e^v)_-,(e^v)_+\})^2}{\min\{
	h_{K_-}\exp(-\|v\|_{L^\infty(K_-)}),h_{K_+}\exp(-\|v\|_{L^\infty(K_+)})\}} \\
	&\ge \frac{p^2}{2{\tt h}}C_{\rm inv}^2(\max\{(e^v)_-,(e^v)_+\})^2
	\max\big\{\exp(\|v\|_{L^\infty(K_-)}),\exp(\|v\|_{L^\infty(K_+)})\big\} \\
	&= \frac13\frac{p^2}{{\tt h}}\alpha(v).
\end{align*}
This shows that
\begin{equation}\label{2.aux2}
  B(v;v,v) \ge 2\sum_{K\in\T_h}\int_K|\na e^{v/2}|^2 dx
  + \frac13\sum_{f\in\E_h}\int_f\frac{p^2}{{\tt h}}\alpha(v)|\jump{v}|^2 ds.
\end{equation}
By the definition of the jumps and the mean-value theorem for $x\in f$,
$$
  |\jump{e^{v/2}}|^2 = \big|e^{v_-/2}-e^{v_+/2}\big|^2
	\le \frac14\max\{e^{v_-},e^{v_+}\}|\jump{v}|^2.
$$
We use definition \eqref{def.alpha} and insert the previous estimate into 
\eqref{2.aux2}:
\begin{align*}
  B(v;v,v) &\ge 2\sum_{K\in\T_h}\int_K|\na e^{v/2}|^2 dx 
	+ 2C_{\rm inv}^2\sum_{f\in\E_h}\int_f\frac{p^2}{{\tt h}}\max\{(e^v)_-,(e^v)_+\} \\
	&\phantom{xx}{}\times
	\max\big\{\exp(\|v\|_{L^\infty(K_-)}),\exp(\|v\|_{L^\infty(K_+)})\big\}
	|\jump{e^{v/2}}|^2 ds.
\end{align*}
Since $(e^{v})_\pm\ge \exp(-\|v\|_{L^\infty(K_\pm)})$, we have
$$
  \max\{(e^v)_-,(e^v)_+\}
	\max\big\{\exp(\|v\|_{L^\infty(K_-)}),\exp(\|v\|_{L^\infty(K_+)})\big\} \ge 1.
$$
This finishes the proof.
\end{proof}

\begin{proposition}[Existence]\label{prop.ex}
Let $\eps>0$. Given $\lambda_h^{k-1}\in V_h$, the DG scheme \eqref{dg} admits a solution
$\lambda_h^k\in V_h$.
\end{proposition}

\begin{proof}
The idea is to apply the Leray-Schauder fixed-point theorem.
We define the fixed-point operator $\Phi:V_h\times[0,1]\to V_h$ by
$\Phi(w,\sigma) = w$, where $v\in V_h$ is the unique solution to the linear problem
\begin{equation}\label{2.lin}
  \eps\int_\Omega v\phi dx
	= \sigma\int_\Omega\big(e^{\lambda_h^{k-1}} - e^w
	+ \triangle t e^w(1-e^w)\big)\phi dx - \sigma\triangle t B(w;w,\phi)
\end{equation}
for $\phi\in V_h$. The left-hand side defines the bilinear form $a(w,\phi)$,
which is coercive, $a(w,w)=\eps\|w\|_{L^2(\Omega)}^2$. The right-hand side
defines a linear form which is continuous on $L^2(\Omega)$ (using the fact
that in finite dimensions, all norms are equivalent). Thus, $\Phi$ is well defined
by the Lax-Milgram lemma. As the right-hand side of \eqref{2.lin}
is continuous with respect to $w$,
standard arguments show that $\Phi$ is continuous. Furthermore, $\Phi(w,0)=0$.
It remains to prove that there exists a uniform bound for all fixed points
of $\Phi$. To this end, let $v\in V_h$ and $\sigma\in[0,1]$ 
such that $\Phi(v,\sigma)=v$.

Let $s(v)=v(\log v-1)+1\ge 0$. The convexity of $s$ implies that 
\begin{equation}\label{2.s}
  (e^{\lambda_h^{k-1}}-e^v)v = (e^{\lambda_h^{k-1}}-e^v)s'(e^v)
	\le s(e^{\lambda_h^{k-1}}) - s(e^v).
\end{equation}
Then, using the test function $\phi=v$ in \eqref{2.lin} gives, because of
the properties $B(v;v,v)\ge 0$ (Lemma \ref{lem.coerc}) and $e^v(1-e^v)v\le 0$,
\begin{align}
  \eps\|v\|_{L^2(\Omega)}^2
	&= \sigma\int_\Omega(e^{\lambda_h^{k-1}}-e^v)v dx
	+ \sigma\triangle t\int_\Omega e^v(1-e^v)v dx
	- \sigma\triangle t B(v;v,v) \nonumber \\
	&\le \sigma\int_\Omega\big(s(e^{\lambda_h^{k-1}}) - s(e^v)\big)dx
	\le \sigma\int_\Omega s(e^{\lambda_h^{k-1}})dx. \label{2.L2}
\end{align}
This is the desired uniform bound. We infer the existence of a solution
to \eqref{dg} by the Leray-Schauder fixed-point theorem.
\end{proof}


\subsection{Discrete entropy inequality and exponential decay}\label{sec.ent}

Let $\lambda_h^k\in V_h$ be a solution to \eqref{dg}. We show that the entropy
$$
  S_h^k := \int_\Omega s(e^{\lambda_h^k})dx, \quad \mbox{where }
	s(u) = u(\log u-1)+1,
$$
is nonincreasing with respect to $k\in\N$.

\begin{lemma}[Discrete entropy inequality]\label{lem.epi}
Let $\eps\ge 0$ and let $\lambda_h^k\in V_h$ be a solution to \eqref{dg}. Then
\begin{equation}\label{2.epi}
  S_h^k + C_0\triangle t\int_\Omega\bigg|e^{\lambda_h^k/2}-\frac{1}{|\Omega|}
	\int_\Omega e^{\lambda_h^k/2}dy\bigg|^2 dx
	+ \triangle t\int_\Omega e^{\lambda_h^k}(e^{\lambda_h^k}-1)\lambda_h^k dx
	\le S_h^{k-1},
\end{equation}
where the constant $C_0>0$ only depends on $C_{\rm inv}$ and $C_{\rm PW}$ from
Lemmas \ref{lem.inv} and \ref{lem.dpw}.
\end{lemma}

\begin{proof}
We take $\phi_h=\lambda_h^k$ as a test function in \eqref{dg} and use
inequality \eqref{2.s} to find that
\begin{align}
  S_h^k - S_h^{k-1} &= \int_\Omega
	\big(s(e^{\lambda_h^k})-s(e^{\lambda_h^{k-1}})\big)dx \nonumber \\
	&= -\triangle t B(\lambda_h^k;\lambda_h^k,\lambda_h^k)
	- \eps\int_\Omega(\lambda_h^k)^2 dx - \triangle t\int_\Omega e^{\lambda_h^k}
	(e^{\lambda_h^k}-1)dx \nonumber \\
	&\le -\triangle t B(\lambda_h^k;\lambda_h^k,\lambda_h^k)
	- \triangle t\int_\Omega e^{\lambda_h^k}(e^{\lambda_h^k}-1)\lambda_h^k \lambda_h^kdx.
	\label{2.epiB}
\end{align}
It remains to estimate the first term on the right-hand side. For this, we use
the coercivity estimate of Lemma \ref{lem.coerc} and the discrete 
Poincar\'e-Wirtinger inequality from Lemma \ref{lem.dpw}:
\begin{align*}
  B(\lambda_h^k;\lambda_h^k,\lambda_h^k)
	&\ge 2\min\{1,C_{\rm inv}^2\}\bigg(\sum_{K\in\T_h}\int_K|\na e^{\lambda_h^k/2}|^2 dx
	+ \sum_{f\in\E_h}\int_f\frac{p^2}{{\tt h}}|\jump{e^{\lambda_h^k/2}}|^2 ds\bigg) \\
  &\ge 2\min\{1,C_{\rm inv}^2\}C_{\rm PW}^{-2}
	\int_\Omega\bigg|e^{\lambda_h^k/2}-\frac{1}{|\Omega|}
	\int_\Omega e^{\lambda_h^k/2}dx\bigg|^2 dx.
\end{align*}
Setting $C_0=2\min\{1,C_{\rm inv}^2\}C_{\rm PW}^{-2}$ finishes the proof.
\end{proof}

We wish to bound the total mass $\int_\Omega \exp(\lambda_h^k)dx$ from below and
above. Since $s(u)=u(\log u-1)+1$ is invertible only on $[0,1]$ and on $[1,\infty)$
but not globally on $[0,\infty)$, we introduce the following functions:
\begin{align*}
  & \sigma_-:[0,\infty)\to[0,1], \quad \sigma_-(v)=(s|_{[0,1]})^{-1}(v) 
	\mbox{ for }v\in[0,1], \ \sigma_-(v)=0\mbox{ for }v\in[1,\infty), \\
	& \sigma_+:[0,\infty)\to[1,\infty), \quad \sigma_+(v)=(s|_{[1,\infty)})^{-1}(v)
	\mbox{ for }v\in[0,\infty).
\end{align*}
In particular, $\sigma_-\circ s=\mbox{id}$ on $[0,1]$ and $\sigma_+\circ s=\mbox{id}$
on $[1,\infty)$. 

\begin{lemma}[Bounds for the total mass]\label{lem.mass}
Let $\eps\ge 0$ and let $\lambda_h^k$ be a solution to \eqref{dg}. Then
$$
  \sigma_-\bigg(\frac{S_h^0}{|\Omega|}\bigg)
	\le \frac{1}{|\Omega|}\int_\Omega e^{\lambda_h^k}dx
	\le \sigma_+\bigg(\frac{S_h^0}{|\Omega|}\bigg).
$$
\end{lemma}

Observe that if $S_h^0<|\Omega|$, the lower bound 
$\sigma_-(S_h^0/|\Omega|)$ is positive.
Thus, the total mass can never vanish, which excludes the case of solutions
converging for $k\to\infty$ to the zero solution. The reason for the difference
between $S_h^0<|\Omega|$ and $S_h^0\ge |\Omega|$ lies in the fact that 
\eqref{sc.eq}-\eqref{sc.bc} admits two steady states, $\lambda^k_h=0$ 
(corresponding to $u^k=e^{\lambda^k_h}=1$) and $\lambda^k_h=-\infty$
(corresponding to $u^k=0$). The assumption $S_h^0<|\Omega|$
will be crucial to prove the decay estimate for the entropy; see
Proposition \ref{prop.decay}. We discuss the case $S_h^0\ge|\Omega|$ in 
Remark \ref{rem.S0}.

\begin{proof}[Proof of Lemma \ref{lem.mass}]
First, we show the lower bound. If $S_h^0\ge|\Omega|$, we have 
$\sigma_-(S_h^0/|\Omega|)=0$, and there is nothing to prove. Thus, let
$S_h^0<|\Omega|$. Set $\beta_k=\min\{1,\exp(\lambda_h^k)\}\le 1$. As $s$ is convex,
we infer from Jensen's inequality and $s(\beta_k)=0$ for $\lambda_h^k>0$ that
$$
  s\bigg(\frac{1}{|\Omega|}\int_\Omega\beta_k dx\bigg)
	\le \frac{1}{|\Omega|}\int_\Omega s(\beta_k)dx 
	= \frac{1}{|\Omega|}\int_{\{\lambda_h^k\le 0\}}s(e^{\lambda_h^k})dx
	\le \frac{S_h^k}{|\Omega|} \le \frac{S_h^0}{|\Omega|},
$$
where in the last step we have used the monotonicity of $k\mapsto S_h^k$.
With this preparation, we are able to verify the lower bound. As $\sigma_-$ is
decreasing, we find that
$$
  \frac{1}{|\Omega|}\int_\Omega e^{\lambda_h^k}dx
	\ge \frac{1}{|\Omega|}\int_\Omega \beta_k dx
	= (\sigma_-\circ s)\bigg(\frac{1}{|\Omega|}\int_\Omega \beta_k dx\bigg)
	\ge \sigma_-\bigg(\frac{S_h^0}{|\Omega|}\bigg).
$$

For the upper bound, we can assume that $\int_\Omega\exp(\lambda_h^k)dx\ge|\Omega|$,
since otherwise, the inequality is trivially satisfied in view of $\sigma_+(v)\ge 1$.
By the concavity of $\sigma_+$, we can again apply the Jensen inequality:
$$
  \sigma_+\bigg(\frac{S^0_h}{|\Omega|}\bigg)
	\ge \sigma_+\bigg(\frac{1}{|\Omega|}\int_\Omega s(e^{\lambda_h^k}) dx\bigg)
	\ge \frac{1}{|\Omega|}\int_\Omega(\sigma_+\circ s)(e^{\lambda_h^k}) dx
	= \frac{1}{|\Omega|}\int_\Omega e^{\lambda_h^k} dx,
$$
proving the claim.
\end{proof}

\begin{proposition}[Discrete entropy decay]\label{prop.decay}
Let $\eps\ge 0$ and let $\lambda_h^k$ be a solution to \eqref{dg}. 
We assume that $S_h^0<|\Omega|$. Then there exists a constant $C_1>0$, only depending
on $S_h^0$, such that for all $k\in\N$,
\begin{equation}\label{2.decay}
  S_h^k \le (1+C_1\triangle t)^{-k}S_h^0.
\end{equation}
In particular, with $\eta = \log(1+C_1\triangle t)/(C_1\triangle t)<1$, 
we have the exponential decay
$$
  S_h^k \le S_h^0 e^{-\eta C_1 k\triangle t}, \quad k\in\N.
$$
\end{proposition}

The proof is based on two properties: The diffusion drives the solution towards
a constant, while the reaction term guarantees that there is only one (positive)
steady state. In order to cope with the interplay of diffusion and reaction, we
prove first the following lemma.

\begin{lemma}\label{lem.aux}
Introduce for $\theta>0$ the functions
$$
  M_1(\theta) = \frac{s(\theta)}{\theta(\theta-1)\log\theta}, \quad
	M_2(\theta) = \max\{1,s(\theta)\}.
$$
Then
$$
  s(e^v) \le \left\{\begin{array}{ll}
	M_1(\theta)e ^v(e^v-1)v \quad & \mbox{if }v\ge\log\theta, \\
	M_2(\theta) \quad & \mbox{if }v<\log\theta.
	\end{array}\right.
$$
\end{lemma}

\begin{proof}
The function
$$
  g(v) = \frac{s(e^v)}{e^v(e^v-1)v} = \frac{e^v(v-1)+1}{e^v(e^v-1)v}, \quad v\neq 0,
$$
can be continuously extended to $v=0$ (with value $g(0)=1/2$) and it is 
decreasing with limits $\lim_{v\to\infty}g(v)=0$ and
$\lim_{v\to-\infty}g(v)=+\infty$. Therefore, $g(v)\le g(\log\theta) = M_1(\theta)$
for all $v\ge\log\theta$, showing the first inequality. For the second one,
let $v\le\log\theta$. Then $s(e^v)<1$ for $v\le 0$ and the monotonicity of $v\mapsto
s(e^v)$ for $v\ge 0$ implies that $s(e^v)\le s(\theta)$. Thus, for any $v\in\R$,
$s(e^v)\le\max\{1,s(\theta)\}=M_2(\theta)$, completing the proof.
\end{proof}

\begin{proof}[Proof of Proposition \ref{prop.decay}]
The idea of the proof is to split $S^k_h$ into two integrals,
\begin{equation}\label{2.twoint}
  S_h^k = \int_{\{\lambda_h^k\le\log\alpha\}}s(e^{\lambda_h^k})dx
	+ \int_{\{\lambda_h^k>\log\alpha\}}s(e^{\lambda_h^k})dx,
\end{equation}
for some suitably chosen $\alpha>0$ and to estimate these integrals by
the second and third terms on the left-hand side of the discrete entropy 
inequality \eqref{2.epi}.

Since $S_h^0/|\Omega|<1$, there exists $\theta\in(0,1)$ such that
$s(\theta)>S_h^0/|\Omega|$. Let
$0<\eps_0<[1-S_h^0/(|\Omega|s(\theta))]^2$ and set $\alpha=\eps_0\theta\in(0,1)$.

We turn to the first integral on the right-hand side of \eqref{2.twoint}.
We claim that there exists a constant $C_{\eps_0 \theta}>0$ such that
\begin{equation}\label{2.claim}
  \int_{\{\lambda_h^k\le\log\alpha\}}s(e^{\lambda_h^k})dx
	\le C_{\eps_0 \theta}\int_\Omega\int_\Omega\bigg|e^{\lambda_h^k/2}
	- \frac{1}{|\Omega|}\int_\Omega e^{\lambda_h^k/2} dy\bigg|^2 dx.
\end{equation}
To prove this inequality, we begin by showing that $\int_\Omega \exp(\lambda_h^k/2)dx$
is bounded from below. Indeed, using the monotonicity of $s\circ\exp$ in $[0,1]$
and of $k\mapsto S^k$,
\begin{align*}
  |\{\lambda_h^k\le\log\theta\}|
	&= |\{s(e^{\lambda_h^k})\ge s(\theta)\}|
	= \frac{1}{s(\theta)}\int_{\{s(\exp(\lambda_h^k))\ge s(\theta)\}}s(\theta)dx \\
	&\le \frac{1}{s(\theta)}\int_{\{s(\exp(\lambda_h^k))\ge s(\theta)\}}
	s(e^{\lambda_h^k})dx
	\le \frac{1}{s(\theta)}\int_\Omega s(e^{\lambda_h^k})dx 
	= \frac{S^k_h}{s(\theta)} \le \frac{S^0_h}{s(\theta)}.
\end{align*}
This yields the lower bound
\begin{align*}
  \int_\Omega e^{\lambda_h^k/2}dx
	&\ge \int_{\{\lambda_h^k>\log\theta\}}e^{\lambda_h^k/2}dx
	> \sqrt{\theta}|\{\lambda_h^k>\log\theta\}| \\
	&= \sqrt{\theta}\big(|\Omega| - |\{\lambda_h^k\le\log\theta\}|\big)
	\ge \sqrt{\theta}\bigg(|\Omega| - \frac{S_h^0}{s(\theta)}\bigg).
\end{align*}
Therefore, as long as $\lambda_h^k\le\log(\eps_0 \theta)$, the difference
$$
  \frac{1}{|\Omega|}\int_\Omega e^{\lambda_h^k/2}dx - e^{\lambda_h^k/2}
	\ge \frac{1}{|\Omega|}\int_\Omega e^{\lambda_h^k/2}dx - \sqrt{\eps_0 \theta}
	\ge \sqrt{\theta}\bigg(1 - \frac{S_h^0}{|\Omega|s(\theta)} 
	- \sqrt{\eps_0 }\bigg) > 0
$$
is positive. Squaring this expression and integrating over 
$\{\lambda_h^k\le\log(\eps_0 \theta)\}$ thus does not change the inequality sign:
\begin{align*}
  \int_{\{\lambda_h^k\le\log(\eps_0 \theta)\}}
	\bigg|e^{\lambda_h^k/2} - \frac{1}{|\Omega|}\int_\Omega e^{\lambda_h^k/2}dx\bigg|^2 dx
	&\ge \int_{\{\lambda_h^k\le\log(\eps_0 \theta)\}}\theta
	\bigg(1 - \frac{S_h^0}{|\Omega|s(\theta)} - \sqrt{\eps_0 }\bigg)^2 dx \\
	&= \big|\{\lambda_h^k\le\log(\eps_0 \theta)\}\big|\theta
	\bigg(1 - \frac{S_h^0}{|\Omega|s(\theta)} - \sqrt{\eps_0 }\bigg)^2.
\end{align*}
Combining the estimate of Lemma \ref{lem.aux} and the previous estimate, we arrive at
\begin{align*}
  \int_{\{\lambda_h^k\le\log(\eps_0 \theta)\}}s(e^{\lambda_h^k})dx
	&\le M_2(\eps_0 \theta)\big|\{\lambda_h^k\le\log(\eps_0 \theta)\}\big| \\
	&\le \frac{M_2(\eps_0 \theta)}{(1-S_h^0/(|\Omega|s(\theta)) - \sqrt{\eps_0 })\theta}
	\int_\Omega\bigg|e^{\lambda_h^k/2} 
	- \frac{1}{|\Omega|}\int_\Omega e^{\lambda_h^k/2}dx\bigg|^2 dx.
\end{align*}
This proves claim \eqref{2.claim} with 
$$
  C_{\eps_0 \theta} = \frac{M_2(\eps_0 \theta)}{(1-S_h^0/(|\Omega|s(\theta)) 
	- \sqrt{\eps_0 })\theta},
$$
recalling that $\alpha=\eps_0 \theta$.

Next, we estimate the second integral on the right-hand side of \eqref{2.twoint}.
It follows from Lemma \ref{lem.aux} that
$$
  \int_{\{\lambda_h^k > \log(\eps_0 \theta)\}}s(e^{\lambda_h^k})dx
	\le M_1(\eps_0 \theta)\int_\Omega e^{\lambda_h^k}(e^{\lambda_h^k}-1)\lambda_h^k dx.
$$
Therefore, \eqref{2.twoint} gives
\begin{align*}
  S_h^k &\le C_{\eps_0 \theta}\int_\Omega\bigg|e^{\lambda_h^k/2} 
	- \frac{1}{|\Omega|}\int_\Omega e^{\lambda_h^k/2}dx\bigg|^2 dx
	+ M_1(\eps_0 \theta)\int_\Omega e^{\lambda_h^k}(e^{\lambda_h^k}-1)\lambda_h^k dx \\
	&\le \frac{1}{C_1}\bigg(C_0\int_\Omega\bigg|e^{\lambda_h^k/2} 
	- \frac{1}{|\Omega|}\int_\Omega e^{\lambda_h^k/2}dx\bigg|^2 dx
	+ \int_\Omega e^{\lambda_h^k}(e^{\lambda_h^k}-1)\lambda_h^k dx\bigg)
\end{align*}
for $C_1 = 1/\max\{C_{\eps_0 \theta}/C_0,M_1(\eps_0 \theta)\}$.
Finally, by Lemma \ref{lem.epi},
$$
  S_h^k \le \frac{1}{C_1\triangle t}(S_h^{k-1}-S_h^{k}),
$$
and solving this recursion shows the proposition.
\end{proof}

\begin{theorem}[Decay in the $L^1(\Omega)$ norm]\label{thm.decay}\ \
Let the assumptions of Proposition \ref{prop.decay} hold. Then
there exists a constant $C_2>0$, only depending on $S_h^0$ and $|\Omega|$, such that
$$
  \|e^{\lambda_h^k}-1\|_{L^1(\Omega)} \le C_2e^{-\eta C_1 k\triangle t/2}, 
	\quad k\in\N,
$$
where $\eta\in(0,1)$ and $C_1>0$ are as in Proposition \ref{prop.decay}. 
\end{theorem}

\begin{proof}
To simplify the notation, we set $u=e^{\lambda^k_h}$ and 
$\bar u=|\Omega|^{-1}\int_\Omega e^{\lambda_h^k}dx$. Then 
the Csisz\'ar-Kullback inequality (see, e.g., \cite[(2.8)]{AMTU00}) gives
$$
  \|u-\bar u\|_{L^1(\Omega)}^2
	\le \frac{2}{|\Omega|}\int_\Omega s\bigg(\frac{u}{\bar u}\bigg)\bar u dx
	= \frac{2\bar u}{|\Omega|}\int_\Omega\big(s(u)-s(\bar u)\big)dx
  \le \frac{2\bar u}{|\Omega|}\int_\Omega s(u)dx,
$$
using the property $s(u)\ge 0$ for all $u\ge 0$. We know from Lemma \ref{lem.mass}
that $\bar u$ is bounded from above by $\sigma_+(S_h^0/|\Omega|)$. Hence,
\begin{equation}\label{2.diff}
  \|u-\bar u\|_{L^1(\Omega)}^2 \le \frac{2}{|\Omega|}
	\sigma_+\bigg(\frac{S_h^0}{|\Omega|}\bigg)S_h^k.
\end{equation}

It remains to show that a similar estimate holds for $|\bar u-1|$.
Since the entropy density $s$ is convex, Jensen's inequality shows that
\begin{equation}\label{2.baru}
  s(\bar u) = s\bigg(\frac{1}{|\Omega|}\int_\Omega e^{\lambda_h^k}dx\bigg)
	\le \frac{1}{|\Omega|}\int_\Omega s(e^{\lambda_h^k})dx = \frac{S_h^k}{|\Omega|}
	\le \frac{S_h^0}{|\Omega|} < 1.
\end{equation}
It holds $s(v)<1$ if and only if $v<e$. Consequently, we have $\bar u<e$.
Applying the elementary inequality
$$
  s(u) \ge \frac{(u-1)^2}{(e-1)^2}\quad\mbox{for all }0\le u\le e
$$
to $u=\bar u$ and using \eqref{2.baru} gives
$$
  |\bar u-1|^2 \le (e-1)^2s(\bar u)\le \frac{(e-1)^2}{|\Omega|}S_h^k.
$$
Thus, combining \eqref{2.diff} and the previous inequality, we conclude that
\begin{align*}
  \|e^{\lambda_h^k}-1\|_{L^1(\Omega)}
	&\le \|u-\bar u\|_{L^1(\Omega)} + \|\bar u-1\|_{L^1(\Omega)} \\
	&\le \bigg\{\bigg(\frac{2}{|\Omega|}\sigma_+\bigg(\frac{S_h^0}{|\Omega|}\bigg)
	\bigg)^{1/2} + (e-1)|\Omega|^{1/2}\bigg\}(S_h^k)^{1/2},
\end{align*}
and the proof follows after applying Proposition \ref{prop.decay}.
\end{proof}

\begin{remark}\label{rem.S0}\rm
We discuss the case $S_h^0\ge|\Omega|$. Fix $\triangle t\in(0,1)$ and $L\in\N$
with $L>1$. Define $\lambda_h^k = (L-k)^+\log(1-\triangle t)$, where 
$z^+=\max\{0,z\}$ denotes the positive part of $z\in\R$. Then
$e^{\lambda_h^k}=(1-\triangle t)^{L-k}<1$ for $k<L$ and $e^{\lambda_h^k}=1$
for $k\ge L$. Consider the case $L>k=1$. Then, setting 
$\delta:=(1-\triangle t)^{L-k}$, we estimate
\begin{align*}
  \frac{1}{\triangle t}& S_h^1
	+ C_0\int_\Omega\bigg|e^{\lambda_h^1/2}-\frac{1}{|\Omega|}\int_\Omega 
	e^{\lambda_h^1/2}dx\bigg|^2 dx + \int_\Omega e^{\lambda^1}(e^{\lambda^1}-1)
	\lambda_h^1 dx \\
	&= \bigg(\frac{s(\delta)}{\triangle t} + \delta(\delta-1)\log\delta\bigg)|\Omega| 
	\le (1 + (1-\triangle t)\delta\log \delta)\frac{|\Omega|}{\triangle t}
	\le \frac{|\Omega|}{\triangle t} \le \frac{S_h^0}{\triangle t}.
\end{align*}
If $1<k\le L$, we deduce from $e^{\lambda_h^k}\le 1$ that
\begin{align}
  \frac{1}{\triangle t}(e^{\lambda_h^k}-e^{\lambda_h^{k-1}})
	&= \frac{1}{\triangle t}\big((1-\triangle t)^{L-k} - (1-\triangle t)^{L-k+1}\big)
	= (1-\triangle t)^{L-k} \nonumber \\
	&= e^{\lambda_h^k} \ge -e^{\lambda_h^k}(e^{\lambda_h^k}-1). \label{2.aux6}
\end{align}
By the convexity of $s$, it follows that $s(u)-s(v)\le (u-v)s'(u) = (u-v)\log u$ for 
all $u$, $v>0$. Since $\lambda_h^k\le 0$ for $k\le L$, \eqref{2.aux6} yields
$$
  s(e^{\lambda_h^k}) \le (e^{\lambda_h^k}-e^{\lambda_h^{k-1}})\lambda_h^k
	+ s(e^{\lambda_h^{k-1}}) \le -\triangle t e^{\lambda_h^k}(e^{\lambda_h^k}-1)
	\lambda_h^k + s(e^{\lambda_h^{k-1}}),
$$
which directly implies the entropy inequality \eqref{2.epi}.
This inequality is trivially satisfied for $k\ge L$. However, it holds 
for $L=2k$ that
$$
  e^{\lambda_h^k} = (1-\triangle t)^k\to 0, \quad
	S_h^k = \int_\Omega s(e^{\lambda_h^k})dx\to |\Omega|\quad\mbox{as }k\to\infty.
$$
This means that if $S_h^0\ge|\Omega|$, there exists no constant $C>0$ depending only
on $S_h^0$ such that \eqref{2.decay} holds for all $(\lambda_h^k)\subset L^2(\Omega)$
satisfying the entropy inequality \eqref{2.epi}. Note that the
constructed function $e^{\lambda_h^k}$ does not possess a uniform 
positive lower bound.
\qed
\end{remark}


\section{Analysis of the DG scheme: numerical convergence}\label{sec.conv}

We show first that the solutions to \eqref{dg} are uniformly bounded in the
DG norm \eqref{dgnorm} if the initial entropy $S_h^0$ is bounded uniformly in $h$.

\begin{lemma}[Uniform bound in DG norm]\label{lem.dgnorm}
Let $\eps\ge 0$ and let $\lambda_h^k$ be a solution to \eqref{dg}.
Then there exists a constant $C>0$ such that
$$
  \triangle t\|e^{\lambda_h^k/2}\|_{\rm DG}^2
	\le 2\triangle t|\Omega| + \max\bigg\{\frac{1}{2\min\{1,C_{\rm inv}^2\}},\triangle t
	\bigg\}S_h^0.
$$
\end{lemma}

\begin{proof}
We have shown in the proof of Lemma \ref{lem.epi} that
$$
  S_h^k + 2\triangle t\min\{1,C_{\rm inv}^2\}
	\bigg(\sum_{K\in\T_h}\int_K|\na e^{\lambda_h^k/2}|^2 dx
	+ \sum_{f\in\E_h}\int_f\frac{p^2}{{\tt h}}|\jump{e^{\lambda_h^k/2}}|^2 ds\bigg)
  \le S_h^{k-1}.
$$
Then, by definition of the DG norm,
$$
  \triangle t\|e^{\lambda_h^k/2}\|_{\rm DG}^2
	\le \triangle t\int_\Omega e^{\lambda_h^k}dx
	+ \frac{1}{2\min\{1,C_{\rm inv}^2\}}\int_\Omega\big(s(e^{\lambda_h^{k-1}})
	- s(e^{\lambda_h^k})\big)dx
$$
Using the inequality $u\le 2+s(u)$ for $u\ge 0$, applied to $u=e^{\lambda_h^k}$,
and the monotonicity of $k\mapsto S_h^k$, we find that
\begin{align*}
  \triangle t\|e^{\lambda_h^k/2}\|_{\rm DG}^2
	&\le \triangle t\int_\Omega(2 + s(e^{\lambda_h^k}))dx
	+ \frac{1}{2\min\{1,C_{\rm inv}^2\}}\int_\Omega\big(s(e^{\lambda_h^{k-1}})
	- s(e^{\lambda_h^k})\big)dx \\
	&= 2\triangle t|\Omega| + \bigg(\triangle t - \frac{1}{2\min\{1,C_{\rm inv}^2\}}\bigg)
	S_h^k + \frac{S_h^{k-1}}{2\min\{1,C_{\rm inv}^2\}} \\
  &\le 2\triangle t|\Omega| + \bigg(\triangle t - \frac{1}{2\min\{1,C_{\rm inv}^2\}}\bigg)
	S_h^k + \frac{S_h^{0}}{2\min\{1,C_{\rm inv}^2\}}.
\end{align*}
If $2\min\{1,C_{\rm inv}^2\}\triangle t\le 1$ then
$$
  \triangle t\|e^{\lambda_h^k/2}\|_{\rm DG}^2
	\le 2\triangle t|\Omega| + \frac{S_h^{0}}{2\min\{1,C_{\rm inv}^2\}}.
$$
On the other hand, if $2\min\{1,C_{\rm inv}^2\}\triangle t > 1$, we have, again
by the monotonicity of $k\mapsto S_h^k$, 
$$
  \bigg(\triangle t - \frac{1}{2\min\{1,C_{\rm inv}^2\}}\bigg)S_h^k
	\le \bigg(\triangle t - \frac{1}{2\min\{1,C_{\rm inv}^2\}}\bigg)S_h^0,
$$
such that in either case,
$$
  \triangle t\|e^{\lambda_h^k/2}\|_{\rm DG}^2
	\le 2\triangle t|\Omega| + \max\bigg\{\frac{1}{2\min\{1,C_{\rm inv}^2\}},\triangle t
	\bigg\}S_0^h,
$$
proving the lemma.
\end{proof}

\begin{theorem}[Convergence]\label{thm.conv}
Let $\eps\ge 0$, $\triangle t\in(0,1)$, and let $\lambda_h^k$ 
be a solution to \eqref{dg}.
Assume that $\lambda_h^{k-1}\in V_h$ such that
$e^{\lambda_h^{k-1}}\to u^{k-1}$ strongly in $L^2(\Omega)$ as $(\eps,h)\to 0$. 
Then there exists a unique strong
solution $u^k\in H^2_n(\Omega)$ to 
\begin{equation}\label{2.uk}
  \frac{1}{\triangle t}(u^k-u^{k-1}) = \Delta u^k + u^k(1-u^k)\quad\mbox{in }\Omega,
	\quad \na u^k\cdot n=0\quad\mbox{on }\pa\Omega
\end{equation}
such that
$$
  e^{\lambda_h^k}\to u^k \quad\mbox{strongly in }L^2(\Omega)
	\mbox{ as }(\eps,h)\to 0.
$$
\end{theorem}

\begin{proof}
Let $\lambda_h^k\in V_h$ be a solution to \eqref{dg}.

{\em Step 1:} We claim that there exists a subsequence $(\eps_i,h_i)\to 0$
such that 
$$
  e^{\lambda_{h_i}^k}\to u^k\quad\mbox{strongly in }L^2(\Omega)\mbox{ as }i\to\infty.
$$
Indeed, by assumption, the initial entropy $(S_{h_i}^0)_{i\in\N}$ is bounded.
Then Lemma \ref{lem.dgnorm} implies that $e^{\lambda_{h}^k/2}$ is bounded in
the DG norm uniformly in $\eps$ and $h$. By the compactness Lemma \ref{lem.comp},
there exists a subsequence $(\eps_i,h_i)\to 0$ and a function $v^k\in H^1(\Omega)$
satisfying 
$$
  e^{\lambda_{h_i}^k/2}\to v^k\quad\mbox{strongly in }L^2(\Omega)\mbox{ as }i\to\infty.
$$
Consequently, $e^{\lambda_{h_i}^k}\to (v^k)^2=:u^k$ strongly in $L^1(\Omega)$.
The discrete entropy inequality \eqref{2.epi} shows that
$$
  \int_\Omega g(e^{2\lambda_h^k})dx
	= \int_\Omega e^{\lambda_{h}^k/2}(e^{\lambda_{h}^k/2}-1)\lambda_h^k dx
$$
is bounded uniformly in $(\eps,h)$, where $g(u)=\sqrt{u}(\sqrt{u}-1)\log u$ for
$u\ge 0$. As the function $g:[0,\infty)\to [0,\infty)$ is continuous and
satisfies $g(u)/u\to\infty$ as $u\to\infty$, we can apply the Theorem of
de la Vall\'ee-Poussin \cite[Theorem 1.3, p.~239]{EkTe72} (for a proof, see
\cite[Section II.2]{Mey66}) to conclude that there exists a subsequence
$e^{2\lambda_{h_i}^k}$ such that $e^{2\lambda_{h_i}^k}\to w^k$ weakly in $L^1(\Omega)$
as $i\to\infty$, for some function $w^k$. We deduce from the strong $L^1$ 
convergence of $e^{\lambda_{h_i}^k}$, possibly for another subsequence, that 
$e^{2\lambda_{h_i}^k}\to (u^k)^2=w^k$ a.e.\ in $\Omega$. This implies that
\begin{equation}\label{2.L2conv}
  e^{2\lambda_{h_i}^k}\to (u^k)^2 \quad\mbox{strongly in }L^1(\Omega),
\end{equation}
thus proving the desired $L^2$ convergence. 

{\em Step 2:} We claim that for any $\phi\in H_n^2(\Omega)\cap C^1(\overline\Omega)$,
it holds that
\begin{align}
  \frac{1}{\triangle t}\int_\Omega e^{\lambda_{h_i}^k}\phi dx
	&{}+ \sum_{K\in\T_{h_i}}\int_K e^{\lambda_{h_i}^k}\na\lambda_{h_i}^k
	\cdot\na\phi dx
	- \sum_{f\in\E_{h_i}}\int_f\jump{\lambda_{h_i}^k}\cdot
	\aver{e^{\lambda_{h_i}^k}\na\phi}ds \nonumber \\
	&{}+ \int_\Omega e^{\lambda_{h_i}^k}\big(e^{\lambda_{h_i}^k}-1\big)\phi dx
	\to \frac{1}{\triangle t}\int_\Omega e^{\lambda^{k-1}_{h_i}}\phi dx
	\quad\mbox{as }i\to\infty. \label{2.conv}
\end{align}
Since $\phi$ does not necessarily belong to $V_h$, 
we cannot use it as a test function in the
weak formulation \eqref{dg}. Therefore, let $P_h:C^0(\overline\Omega)\to
C^0(\overline\Omega)\cap V_h$ be the interpolation operator, defined, e.g., in
\cite[Section 2.3]{Cia78}. It possesses the following property 
\cite[Section 3.1.6]{Cia78}: There exists a
constant $C_I>0$ such that for all $K\in\T_h$ and $\phi\in H^2(K)$,
\begin{equation}\label{interpol}
  \|\phi-P_h\phi\|_{W^{m,q}(K)} \le C_I h_K^{2-d/2-(m-d/q)}\|\phi\|_{H^2(K)}
\end{equation}
for $m\le 2\le q$ such that $m-d/q\le 2-d/2$. In particular, for $\phi\in H^2(\Omega)$
and $d\le 3$,
\begin{equation}\label{interpol2}
  \|\phi-P_h\phi\|_{L^{\infty}(\Omega)} \le C_I h_i^{2-d/2}\|\phi\|_{H^2(\Omega)}\to 0
	\quad\mbox{as }h_i\to 0.
\end{equation}
For given $\phi\in H_n^2(\Omega)\cap C^1(\overline\Omega)$,
we choose the test function $\phi_{h_i}:=P_{h_i}\phi$ in \eqref{dg}:
\begin{align}
  \frac{1}{\triangle t}\int_\Omega e^{\lambda_{h_i}^{k-1}}\phi_{h_i}dx 
	&= \frac{1}{\triangle t}\int_\Omega e^{\lambda_{h_i}^k}\phi_{h_i}dx
	+ \eps_i\int_\Omega \lambda_{h_i}^k\phi_{h_i}dx
	+ \sum_{K\in\T_{h_i}}\int_Ke^{\lambda_{h_i}^k}\na\lambda_{h_i}^k
	\cdot\na\phi_{h_i}dx \nonumber \\
	&\phantom{xx}{}- \sum_{f\in\E_{h_i}}\int_f\jump{\lambda_{h_i}^k}\cdot
	\aver{e^{\lambda_{h_i}^k}\na\phi_{h_i}}ds
	+ \int_\Omega e^{\lambda_{h_i}^k}\big(e^{\lambda_{h_i}^k}-1\big)
	\phi_{h_i}dx. \label{2.15}
\end{align}
Here, we have used the fact that $\jump{\phi_{h_i}}=0$ since $\phi_{h_i}$ is
continuous. Note that \eqref{interpol2} implies that $\phi_{h_i}\to \phi$
strongly in $L^\infty(\Omega)$ as $i\to\infty$.
As $e^{\lambda_{h_i}^{k-1}}\to u^{k-1}$ strongly in
$L^2(\Omega)$, by assumption, we have for the left-hand side of \eqref{2.15}:
$$
  \int_\Omega e^{\lambda_{h_i}^{k-1}}(\phi_{h_i}-\phi)dx \to 0\quad\mbox{as }h_i\to 0.
$$
Similarly, as $e^{\lambda_{h_i}^k}\to u^k$ strongly in $L^2(\Omega)$,
we infer for the first and last integrals on the right-hand side of \eqref{2.15} that 
$$
  \int_\Omega e^{\lambda_{h_i}^k}(\phi_{h_i}-\phi)dx\to 0, \quad
	\int_\Omega e^{\lambda_{h_i}^k}\big(e^{\lambda_{h_i}^k}-1\big)
	(\phi_{h_i}-\phi)dx\to 0.
$$
Inequality \eqref{2.L2} shows that
$$
  \eps_i\|\lambda_{h_i}^k\|_{L^2(\Omega)}^2 
	\le \int_\Omega s(e^{\lambda_{h_i}^{k-1}})dx.
$$
Thus, $(\eps_i^{1/2}\lambda_{h_i}^k)$ is bounded in $L^2(\Omega)$ from which we have
$\eps_i\lambda_{h_i}^k\to 0$ strongly in $L^2(\Omega)$ as $(\eps_i,h_i)\to 0$. 
This implies that
the second integral on the right-hand side of \eqref{2.15} converges to zero.

Next, we prove for the third integral on the right-hand side of \eqref{2.15} that
$$
  \sum_{K\in\T_{h_i}}\int_Ke^{\lambda_{h_i}^k}\na\lambda_{h_i}^k
	\cdot\na(\phi_{h_i}-\phi)dx \to 0\quad\mbox{as }h_i\to 0.
$$
Indeed, by the H\"older inequality, the interpolation property \eqref{interpol}, and
the discrete entropy inequality \eqref{2.epi}, we obtain
\begin{align*}
  \bigg|\sum_{K\in\T_{h_i}}&\int_Ke^{\lambda_{h_i}^k}\na\lambda_{h_i}^k
	\cdot\na(\phi_{h_i}-\phi)dx\bigg|
	\le 2\sum_{K\in\T_{h_i}}\|e^{\lambda_{h_i}^k/2}\|_{L^4(K)}
	\|\na e^{\lambda_{h_i}^k/2}\|_{L^2(K)}\|\phi_{h_i}-\phi\|_{W^{1,4}(K)} \\
	&\le 2C_I\sum_{K\in\T_{h_i}}\|e^{\lambda_{h_i}^k/2}\|_{L^4(K)}
	\|\na e^{\lambda_{h_i}^k/2}\|_{L^2(K)}h_K^{1-d/4}\|\phi\|_{H^2(K)} \\
	&\le 2C_I\|e^{\lambda_{h_i}^k}\|_{L^2(\Omega)}^{1/2}
	\bigg(\sum_{K\in\T_{h_i}}\|\na e^{\lambda_{h_i}^k/2}\|_{L^2(K)}^2\bigg)^{1/2}
	\bigg(\sum_{K\in\T_{h_i}}h_K^{2-d/2}\|\phi\|_{H^2(K)}^2\bigg)^{1/2} \\
	&\le Ch_i^{1-d/4}\|\phi\|_{H^2(\Omega)}\to 0.
\end{align*}

It remains to prove for the fourth integral on the right-hand side of \eqref{2.15} that
$$
  \sum_{f\in\E_{h_i}}\int_f\jump{\lambda_{h_i}^k}\cdot
	\aver{e^{\lambda_{h_i}^k}\na(\phi_{h_i}-\phi)}ds \to 0\quad\mbox{as }h_i\to 0.
$$
To this end, we use the elementary inequality 
$|\aver{u\na v}|\le 2\aver{u}\aver{|\na v|}$ for functions $u$, $v$ with
nonnegative $u$ and the Cauchy-Schwarz inequality:
\begin{align}
  \bigg|\sum_{f\in\E_{h_i}}&\int_f\jump{\lambda_{h_i}^k}\cdot
	\aver{e^{\lambda_{h_i}^k}\na(\phi_{h_i}-\phi)}ds\bigg|^2
	\le \bigg|\sum_{f\in\E_{h_i}}\int_f |\jump{\lambda_{h_i}^k}|
	\big|\aver{e^{\lambda_{h_i}^k}\na(\phi_{h_i}-\phi)}\big|ds\bigg|^2 \nonumber \\
	&\le 4\bigg|\sum_{f\in\E_{h_i}}|\jump{\lambda_{h_i}^k}|
	\aver{e^{\lambda_{h_i}^k}}\aver{|\na(\phi_{h_i}-\phi)|}ds\bigg|^2 \nonumber \\
	&\le 4\sum_{f\in\E_{h_i}}\int_f\aver{e^{\lambda_{h_i}^k}}^2 
	|\jump{\lambda_{h_i}^k}|^2 ds
	\sum_{f\in\E_{h_i}}\int_f\aver{|\na(\phi_{h_i}-\phi)|}^2 ds. \label{2.aux3}
\end{align}
We estimate both integrals separately. First, the multiplicative trace inequality
in Lemma~\ref{lem.mti} shows that, for some constant $C>0$ and for faces or edges
$f=\pa K_+\cap\pa K_-$,
$$
  \int_f\aver{|\na(\phi_{h_i}-\phi)|}^2 ds
	\le C\sum_{K=K_\pm}\|\phi_{h_i}-\phi\|_{H^1(K)}\bigg(\frac{1}{h_K}
	\|\phi_{h_i}-\phi\|_{H^1(K)} + \|\phi_{h_i}-\phi\|_{H^2(K)}\bigg).
$$
We deduce from \eqref{interpol}, i.e.
$$
  \|\phi_{h_i}-\phi\|_{H^1(K)} \le C_Ih_K\|\phi\|_{H^2(K)}, \quad
	\|\phi_{h_i}-\phi\|_{H^2(K)} \le C_I\|\phi\|_{H^2(K)},
$$
that
$$
  \sum_{f\in\E_{h_i}}\int_f\aver{|\na(\phi_{h_i}-\phi)|}^2 ds \le Ch_i.
$$
Therefore, also using ${\tt h}(x)\le h_i$, we deduce from \eqref{2.aux3} that
$$
  \bigg|\sum_{f\in\E_{h_i}}\int_f\jump{\lambda_{h_i}^k}\cdot
	\aver{e^{\lambda_{h_i}^k}\na(\phi_{h_i}-\phi)}ds\bigg|^2
	\le C\frac{h_i^2}{p^2}\sum_{f\in\E_{h_i}}\int_f\frac{p^2}{{\tt h}_i}
	\aver{e^{\lambda_{h_i}^k}}^2|\jump{\lambda_{h_i}^k}|^2 ds,
$$
where ${\tt h}_i(x)=\min\{h_{i,K_+},h_{i,K_-}\}$ for $x\in \pa K_+\cap \pa K_-$.
We claim that the sum on the right-hand side is bounded uniformly in $h_i$.
By Definition \eqref{def.alpha},
$$
  \aver{e^{\lambda_{h_i}^k}}^2 \le \frac{2}{3C_{\rm inv}^2}\alpha(\lambda_{h_i}^k),
$$
such that we can estimate
\begin{align*}
  \sum_{f\in\E_{h_i}}\int_f\frac{p^2}{{\tt h}_i}
	\aver{e^{\lambda_{h_i}^k}}^2|\jump{\lambda_{h_i}^k}|^2 ds
	&\le \frac{2}{3C_{\rm inv}^2}\sum_{f\in\E_{h_i}}\int_f
	\frac{p^2}{{\tt h}_i}\alpha(\lambda_{h_i}^k)|\jump{\lambda_{h_i}^k}|^2 ds \\
	&\le \frac{2}{C_{\rm inv}^2}B(\lambda_{h_i}^k;\lambda_{h_i}^k,\lambda_{h_i}^k),
\end{align*}
where we used \eqref{2.aux2} in the last step. 
The proof of Lemma \ref{lem.epi} shows that
$B(\lambda_{h_i}^k;\lambda_{h_i}^k,\lambda_{h_i}^k)\ge C/\triangle t$ since
$e^{\lambda_{h_i}^k}$ is uniformly bounded in $L^2(\Omega)$. We conclude that
$$
  \bigg|\sum_{f\in\E_{h_i}}\int_f\jump{\lambda_{h_i}^k}\cdot
	\aver{e^{\lambda_{h_i}^k}\na(\phi_{h_i}-\phi)}ds\bigg| 
	\le \frac{Ch_i}{(\triangle t)^{1/2}} \to 0 \quad\mbox{as }h_i\to 0.
$$

We put together all the previous convergence results to infer that
\begin{align}
  \frac{1}{\triangle}\int_\Omega & e^{\lambda_{h_i}^k}(\phi_{h_i}-\phi)dx
	+ \eps_i\int_\Omega\lambda_{h_i}^k\phi_{h_i}dx
	+ \sum_{K\in\T_{h_i}}\int_K e^{\lambda_{h_i}^k}\na\lambda_{h_i}^k
	\cdot\na(\phi_{h_i}-\phi)dx \nonumber \\
	&{}- \sum_{f\in\E_{h_i}}\int_f\jump{\lambda_{h_i}^k}\cdot
	\aver{e^{\lambda_{h_i}^k}\na(\phi_{h_i}-\phi)}ds
	+ \int_\Omega e^{\lambda_{h_i}^k}\big(e^{\lambda_{h_i}^k}-1\big)
	(\phi_{h_i}-\phi)dx \nonumber \\
	&{}- \frac{1}{\triangle}\int_\Omega e^{\lambda_{h_i}^{k-1}}(\phi_{h_i}-\phi)dx \to 0
	\quad\mbox{as }i\to\infty. \label{2.star}
\end{align}
Thus, inserting \eqref{2.15}, all integrals involving $\phi_{h_i}$ cancel, and
we end up with \eqref{2.conv}.

{\em Step 3:} We prove that the limit $u^k$, derived in Step 1, is a solution to
the very weak formulation
\begin{equation}\label{2.step3}
  \frac{1}{\triangle t}\int_\Omega(u^k - u^{k-1})\phi dx
	= \int_\Omega u^k\Delta\phi dx + \int_\Omega u^k(1-u^k)\phi dx
\end{equation}
for all $\phi\in H_n^2(\Omega)\cap C^1(\overline\Omega)$. 
For the proof, we pass to the limit $h_i\to 0$ in each term of \eqref{2.conv}.
Because of \eqref{2.L2conv}, we have
$$
  \int_\Omega e^{\lambda_{h_i}^k}\phi dx \to \int_\Omega u^k\phi dx, \quad
  \int_\Omega e^{\lambda_{h_i}^k}\big(e^{\lambda_{h_i}^k}-1\big)\phi dx
	\to \int_\Omega u^k(u^k-1)\phi dx.
$$
The limit $i\to\infty$ in the remaining expression
$$
  I_i := \sum_{K\in\T_{h_i}}\int_K e^{\lambda_{h_i}^k}\na\lambda_{h_i}^k
	\cdot\nabla \phi dx
	- \sum_{f\in\E_{h_i}}\int_f\jump{\lambda_{h_i}^k}\cdot
	\aver{e^{\lambda_{h_i}^k}\na\phi}ds
$$
is more delicate. Consider the first term in the definition of $I_i$. 
Integrating by parts elementwise gives
\begin{align*}
  \sum_{K\in\T_{h_i}}\int_K e^{\lambda_{h_i}^k}\na\lambda_{h_i}^k\cdot\na \phi dx
  &= \sum_{K\in\T_{h_i}}\int_K \na e^{\lambda_{h_i}^k}\cdot\na\phi dx \\
	&= -\sum_{K\in\T_{h_i}}\int_K e^{\lambda_{h_i}^k}\Delta\phi dx
	+ \sum_{K\in\T_{h_i}}\int_{\pa K} e^{\lambda_{h_i}^k}\na\phi\cdot n ds \\
  &= -\int_\Omega e^{\lambda_{h_i}^k}\Delta\phi dx
	+ \sum_{f\in\E_{h_i}}\int_{f} \jump{e^{\lambda_{h_i}^k}}\cdot \na\phi ds,
\end{align*}
where we have used the fact that $\na \phi$ has a continuous normal component across 
interelement boundaries. From the previous identity and the $L^2$ convergence of 
$e^{\lambda^k_{h_i}}$, we obtain
$$
  \sum_{K\in\T_{h_i}}\int_K e^{\lambda_{h_i}^k}\na\lambda_{h_i}^k\cdot\na\phi dx
	- \sum_{f\in\E_{h_i}}\int_f\jump{e^{\lambda_{h_i}^k}}\cdot\na\phi ds
	= -\int_\Omega e^{\lambda_{h_i}^k}\Delta\phi dx  
	\to -\int_\Omega u^k \Delta\phi dx.
$$
We claim that
\begin{equation}\label{2.claim2}
  \sum_{f\in\E_{h_i}}\int_f\jump{e^{\lambda_{h_i}^k}}\cdot\na\phi ds
	- \sum_{f\in\E_{h_i}}\int_f\jump{\lambda_{h_i}^k}\cdot\aver{e^{\lambda_{h_i}^k}
	\na\phi} ds \to 0,
\end{equation}
since this implies that
$$
  I_i \to -\int_\Omega u^k\Delta\phi dx,
$$
and thus shows \eqref{2.step3}. 

For the proof of \eqref{2.claim2}, let $x\in\pa K_+\cap\pa K_-$ for two neighboring
elements $K_+$, $K_-\in\T_{h_i}$ and set $\lambda_\pm:=\lambda_{h_i}^k|_{K_\pm}$.
We assume without loss of generality that $\lambda_+\ge\lambda_-$ since
otherwise, we may exchange $K_+$ and $K_-$. The definitions of the jump
$\jump{\cdot}$ and average $\aver{\cdot}$ imply that
\begin{align}
  \Big|\jump{e^{\lambda_{h_i}^k}}&\cdot\na\phi
	- \jump{\lambda_{h_i}^k}\cdot\aver{e^{\lambda_{h_i}^k}\na\phi}\Big| 
	\nonumber \\
	&= \bigg|\bigg((e^{\lambda_+}n_+ + e^{\lambda_-}n_-)
	- (\lambda_+n_+ + \lambda_-n_-)\frac12(e^{\lambda_+}+e^{\lambda_-})\bigg)
	\cdot\na\phi\bigg| \nonumber \\
	&= e^{\lambda_-}\bigg|\bigg((e^{\lambda_+-\lambda_-}-1)n_+
	- (\lambda_+ - \lambda_-)n_+\frac12(e^{\lambda_+-\lambda_-}+1)\bigg)\cdot\na\phi\bigg| 
	\nonumber \\
	&\le e^{\lambda_-}\bigg|(e^{\lambda_+-\lambda_-}-1) - (\lambda_+ - \lambda_-)
	\frac12(e^{\lambda_+-\lambda_-}+1)\bigg||\na\phi| \nonumber \\
	&=: e^{\lambda_-}|g(\lambda_+-\lambda_-)||\na\phi|, \label{2.aux4}
\end{align}
where $g(s)=(e^s-1)+s(e^s+1)/2$ for $s\ge 0$. A Taylor expansion shows that
$g(s)=g''(\xi)s^2/2=-\xi e^\xi s^2/4$ for some $0\le\xi\le s$. Therefore
$|g(s)|\le s^2 e^{2s}/4$ for $s\ge 0$, and we obtain
\begin{align*}
  e^{\lambda_-}|g(\lambda_+-\lambda_-)|
	&\le \frac{e^{\lambda_-}}{4}(\lambda_+ - \lambda_-)^2 e^{2(\lambda_+ - \lambda_-)}
	= \frac14(\lambda_+ - \lambda_-)^2e^{2\lambda_+ - \lambda_-} \\
  &\le \frac12(\lambda_+ - \lambda_-)^2\frac{e^{2\lambda_+}+e^{2\lambda_-}}{2}
	e^{|\lambda_-|}.
\end{align*}
The difference can be identified with the jump of $\lambda_{h_i}^k$ across
$f=\pa K_+\cap\pa K_-$, while the sum corresponds to the average of
$e^{2\lambda_{h_i}^k}$ in $f$. Thus, it follows from \eqref{2.aux4} that
\begin{align*}
  \bigg|\sum_{f\in\E_{h_i}}&\int_f\jump{e^{\lambda_{h_i}^k}}\cdot\na\phi ds
	- \sum_{f\in\E_{h_i}}\int_f\jump{\lambda_{h_i}^k}\cdot\aver{e^{\lambda_{h_i}^k}
	\na\phi} ds\bigg| \\
	&\le \frac12\sum_{f\in\E_{h_i},\,f=\pa K_+\cap \pa K_-}
	\|\na\phi\|_{L^\infty(f)}\max\big\{\exp(\|\lambda_{h_i}^k\|_{L^\infty(K_+)}),
	\exp(\|\lambda_{h_i}^k\|_{L^\infty(K_-)})\big\} \\
	&\phantom{xx}{}\times\int_f
	\jump{\lambda_{h_i}^k}^2\aver{e^{2\lambda_{h_i}^k}}ds.
\end{align*}
By definition \eqref{def.alpha} of the stabilization factor, it holds that
$$
  \max\big\{\exp(\|\lambda_{h_i}^k\|_{L^\infty(K_+)}),
	\exp(\|\lambda_{h_i}^k\|_{L^\infty(K_-)})\big\}\aver{e^{2\lambda_{h_i}^k}}
	\le \frac{2\alpha(\lambda_{h_i}^k)}{3C_{\rm inv}^2}.
$$
Using this estimate and the coercivity estimate \eqref{2.aux2} for the form $B$, 
we can write
\begin{align*}
  \bigg|\sum_{f\in\E_{h_i}}&\int_f\jump{e^{\lambda_{h_i}^k}}\cdot\na\phi ds
	- \sum_{f\in\E_{h_i}}\int_f\jump{\lambda_{h_i}^k}\cdot\aver{e^{\lambda_{h_i}^k}
	\na\phi} ds\bigg| \\
	&\le \frac{1}{3C_{\rm inv}^2}\|\na\phi\|_{L^\infty(\Omega)}
	\sum_{f\in\E_{h_i}}\int_f\alpha(\lambda_{h_i}^k)|\jump{\lambda_{h_i}^k}|^2 ds \\
	&\le \frac{1}{3C_{\rm inv}^2}\|\na\phi\|_{L^\infty(\Omega)}
	\frac{h_i}{p^2}\sum_{f\in\E_{h_i}}\int_f\frac{p^2}{{\tt h}_i}
	\alpha(\lambda_{h_i}^k)|\jump{\lambda_{h_i}^k}|^2 ds \\
  &\le \frac{h_i}{p^2C_{\rm inv}^2}\|\na\phi\|_{L^\infty(\Omega)}
	B(\lambda_{h_i}^k;\lambda_{h_i}^k,\lambda_{h_i}^k).
\end{align*}
We know from the proof of Lemma \ref{lem.epi} that
$B(\lambda_{h_i}^k;\lambda_{h_i}^k,\lambda_{h_i}^k)\le C/\triangle t$ is
uniformly bounded. This proves our claim \eqref{2.claim2}.

Now, we can pass to the limit $i\to\infty$ in \eqref{2.conv}, which yields 
\eqref{2.step3}. 

{\em Step 4:} We claim that the solution $u^k\in L^2(\Omega)$ 
to \eqref{2.conv} satisfies the
regularity $u^k\in H^1(\Omega)$ and hence is a weak solution to 
\eqref{sc.eq}-\eqref{sc.bc}. To this end, we use the duality method as in
\cite[p.~318]{Bre11}. Let $\T:L^2(\Omega)\to H_n^2(\Omega)$ be defined
by $\T v=u$, where $u$ solves the elliptic problem 
$u-\triangle t\Delta u=v$ in $\Omega$,
$\na u\cdot n=0$ on $\pa\Omega$. By \cite[Theorem 8.3.10]{MaRo10},
for $v\in C_0^\infty(\Omega)$, it holds that $\T v\in H^2_n(\Omega)\cap 
C^1(\overline\Omega)$. Then, introducing $g:=u^{k-1}+\triangle t
u^k(1-u^k)$, the very weak formulation \eqref{2.step3} 
can be equivalently written as
$$
  \int_\Omega u^k(\phi - \triangle t\Delta \phi)dx 
	= \int_\Omega g\phi dx
$$
for all $\phi\in H^2_n(\Omega)\cap C^1(\overline\Omega)$. Given 
$v\in C_0^\infty(\Omega)$, we set $\phi=\T v$, and the previous
equation becomes
\begin{equation}\label{2.dual}
  \int_\Omega u^k v dx = \int_\Omega g\T v dx.
\end{equation}
As $C_0^\infty(\Omega)$ is dense in $L^2(\Omega)$ and $\T$ is continuous, 
\eqref{2.dual} remains valid for all $v\in L^2(\Omega)$. 

Next, we denote by $\T':H_n^2(\Omega)' \to L^2(\Omega)$ the dual operator of $\T$.
According to \cite[Theorem 8.3.10]{MaRo10}, the operator $\T$ can be
extended to an operator $\T:L^p(\Omega)\cap H^1(\Omega)'\to W^{2,p}(\Omega)$
for $1<p\le 2$. (This is basically a regularity statement for the elliptic problem.)
We deduce from the Sobolev embedding theorem that $W^{2,p}(\Omega)\hookrightarrow
C^0(\overline\Omega)$ for $p>3/2$ since $d\le 3$. Therefore, there
exists an extension $\T^*:C^0(\overline\Omega)'\to L^{p'}(\Omega)$ of $\T'$,
where $p'=p/(p-1)<3$. 

Now, $g\in L^1(\Omega)\subset C^0(\overline\Omega)'$. Then \eqref{2.dual}
implies that $u^k=\T^*(g)\in L^{p'}(\Omega)$ for $p'<3$ and consequently,
$g\in L^q(\Omega)$ for $q<3/2$.

It remains to show that $u^k\in H^1(\Omega)$. Since $u^k\in L^2(\Omega)$,
the elliptic problem
$$
  v_m - \frac{1}{m}\Delta v_m = u^k \quad\mbox{in }\Omega, \quad
	\na v_m\cdot n = 0 \quad\mbox{on }\pa\Omega,
$$
possesses a unique solution $v_m\in H^2_n(\Omega)$ \cite[Theorem 8.3.10]{MaRo10}.
Multiplying the elliptic equation by $v_m$ and applying the Cauchy-Schwarz
inequality, we have
$$
  \frac12\int_\Omega v_m^2 dx + \frac{1}{m}\int_\Omega|\na v_m|^2 dx
	\le \frac12\int_\Omega (u^k)^2 dx.
$$
Thus, $(v_m)$ is bounded in $L^2(\Omega)$ and it follows the existence of a
subsequence which is not relabeled that
$v_m\rightharpoonup u^k$ weakly in $L^2(\Omega)$ as $m\to\infty$. 
Using $v=v_m-\triangle t\Delta v_m$ in \eqref{2.dual}, it follows that
\begin{align}
  \int_\Omega g \T v dx &= \int_\Omega u^k v dx
	= \int_\Omega\bigg(v_m - \frac{1}{m}\Delta v_m\bigg)(v_m-\triangle t\Delta v_m)dx 
	\nonumber \\
	&= \int_\Omega v_m^2 dx + \bigg(\triangle t + \frac{1}{m}\bigg)
	\int_\Omega|\na v_m|^2 dx + \frac{\triangle t}{m}\int_\Omega(\Delta v_m)^2 dx
	\nonumber \\
	&\ge \int_\Omega v_m^2 dx + \triangle t\int_\Omega|\na v_m|^2 dx
	\ge \triangle t\|v_m\|_{H^1(\Omega)}^2. \label{2.vm}
\end{align}
We apply the H\"older inequality and use the Sobolev embedding $H^1(\Omega)
\hookrightarrow L^{q'}(\Omega)$ for $q'\le 6$, knowing that 
$g\in L^q(\Omega)$ for $q<3/2$:
\begin{align*}
  \int_\Omega g\T v dx 
	&\le \|g\|_{L^q(\Omega)}\|\T v\|_{L^{q'}(\Omega)}
	\le C\|g\|_{L^q(\Omega)}\|\T v\|_{H^1(\Omega)} \\
	&\le C(\triangle t)\|g\|_{L^q(\Omega)}^2 
	+ \frac{\triangle t}{2}\|v_m\|_{H^1(\Omega)}^2,
\end{align*}
where $3<q'\le 6$ and $1/q+1/q'=1$. 
The $H^1(\Omega)$ norm can be absorbed by the corresponding
term on the right-hand side of \eqref{2.vm}, and we end up with
$$
  \frac{\triangle t}{2}\|v_m\|_{H^1(\Omega)}^2
	\le C(\triangle t)\|g\|_{L^q(\Omega)}^2.
$$
This shows that $(v_m)$ is bounded in $H^1(\Omega)$. Thus, there exists a
subsequence (not relabeled) which converges weakly in $H^1(\Omega)$ to some function 
$w\in H^1(\Omega)$. Since $v_m\rightharpoonup u^k$ 
weakly in $L^2(\Omega)$, we conclude that $w=u^k\in H^1(\Omega)$. 

Knowing that $u^k\in H^1(\Omega)$ solves \eqref{2.step3}, we can integrate
by parts in the first term of the right-hand side of \eqref{2.step3}, leading to
$$
  \frac{1}{\triangle t}\int_\Omega(u^k-u^{k-1})\phi dx
	= -\int_\Omega\na u^k\cdot\na \phi dx + \int_\Omega u^k(1-u^k)\phi dx
$$
for all $\phi\in H^2_n(\Omega)$ and, by density, for all $\phi\in H^1(\Omega)$.

Moreover, since $u^k\in L^4(\Omega)$ and consequently, $u^k(1-u^k)\in L^2(\Omega)$,
elliptic regularity implies that $u^k\in H^2(\Omega)$ and $\na u^k\cdot n=0$
on $\pa\Omega$, i.e.\ $u^k\in H_n^2(\Omega)$. We conclude that $u^k$ solves
\eqref{2.uk}.

{\em Step 5:} We prove the uniqueness of weak solutions
to \eqref{sc.eq}-\eqref{sc.bc} to conclude the convergence of the whole
sequence $e^{\lambda_{h}^k}$ to $u^k$. Let $u^k$, $v^k\in H^1(\Omega)$ be two
weak solutions to \eqref{sc.eq}-\eqref{sc.bc}. Taking the difference of the
corresponding weak formulations with the test function $u^k-v^k$, we obtain
\begin{align*}
  0 &= \frac{1}{\triangle t}\int_\Omega(u^k-v^k)^2 dx
	+ \int_\Omega|\na(u^k-v^k)|^2 dx + \int_\Omega\big(u^k(u^k-1)-v^k(v^k-1)\big)
	(u^k-v^k) dx \\
	&= \bigg(\frac{1}{\triangle t}-1\bigg)\int_\Omega(u^k-v^k)^2 dx
	+ \int_\Omega|\na(u^k-v^k)|^2 dx + \int_\Omega(u^k+v^k)(u^k-v^k)^2 dx.
\end{align*}
Thus, choosing $\triangle t<1$, we infer that $u^k-v^k=0$ in $\Omega$.
\end{proof}


\section{Numerical results}\label{sec.num}

We present some numerical results for the Fisher-KPP equation in one space
dimension,
\begin{align}
  & \pa_t u = Du_{xx} + u(1-u) \quad\mbox{in }\Omega=(0,1),\ t>0, \label{3.eq} \\
	& u_x\cdot n=0\quad\mbox{at }x=0,1,\ t>0, \quad u(0)=u^0\quad\mbox{in }(0,1).
	\label{3.bic}
\end{align}

\subsection{One group of species}

Let $D=10^{-4}$ and $u_0(x)=0.8$ for $0<x<1/2$, $u_0(x)=0$ else.
Problem \eqref{3.eq}-\eqref{3.bic} models the evolution of one species
initially concentrated in the domain $(0,1/2)$. We solve this problem
by using an implicit Euler scheme in time and a continuous $P_1$ finite-element
discretization, both on a uniform mesh. The reaction term is treated implicitly.
The Newton method with relaxation is used at each time step, up to convergence. 
The integrals are computed by using a Gau\ss{}-Legendre quadrature formula 
of order 8.
Figure \ref{fig.u} shows the density $u(x,t)$ at various time instances.
We observe that the finite-element solution $u_h^k$ becomes negative even on the
finer mesh, and it is pushed towards $-\infty$ in some region
since $u^*=0$ is a repulsive steady state. 

\begin{figure}[ht]
\includegraphics[width=0.49\textwidth]{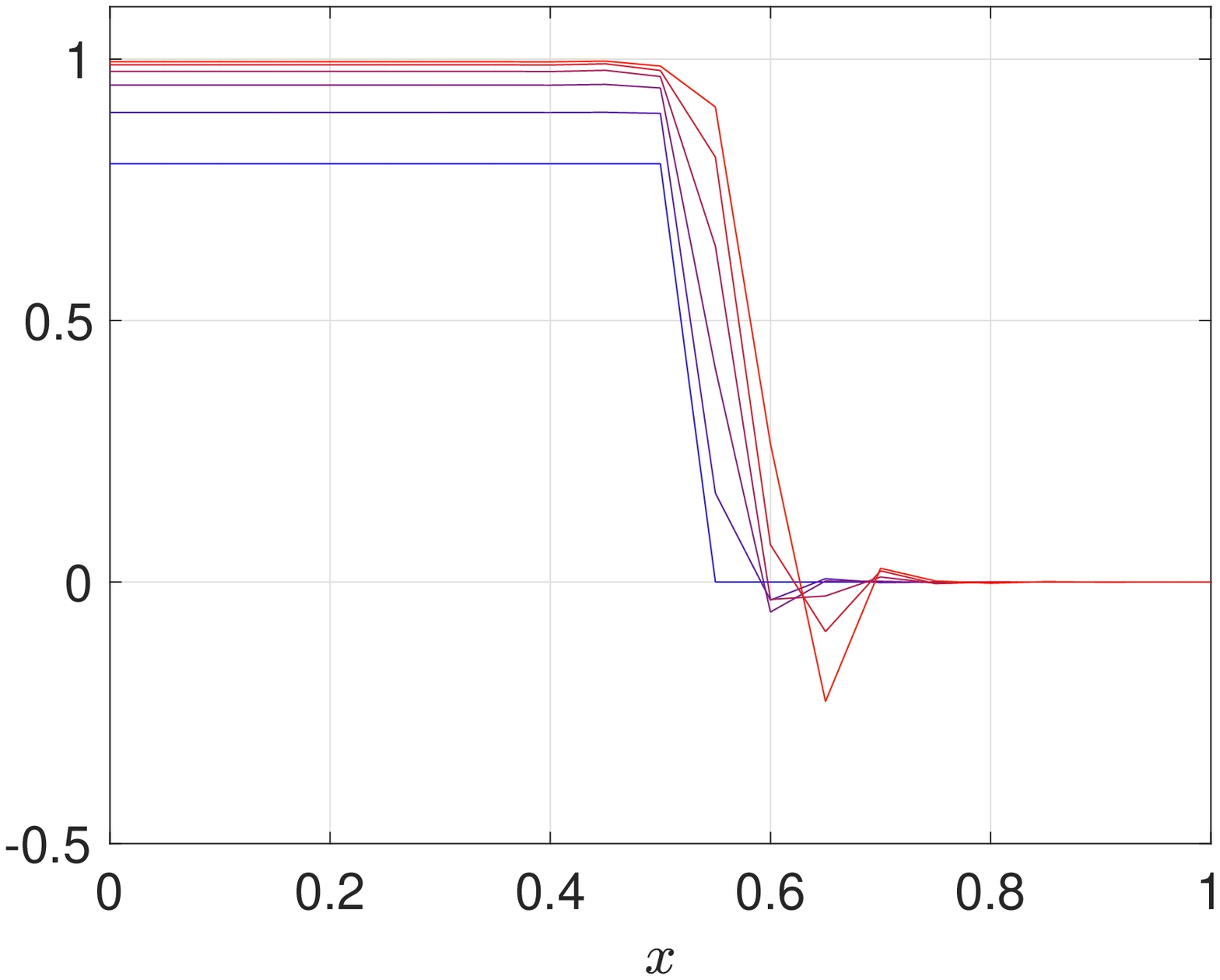}  
\includegraphics[width=0.49\textwidth]{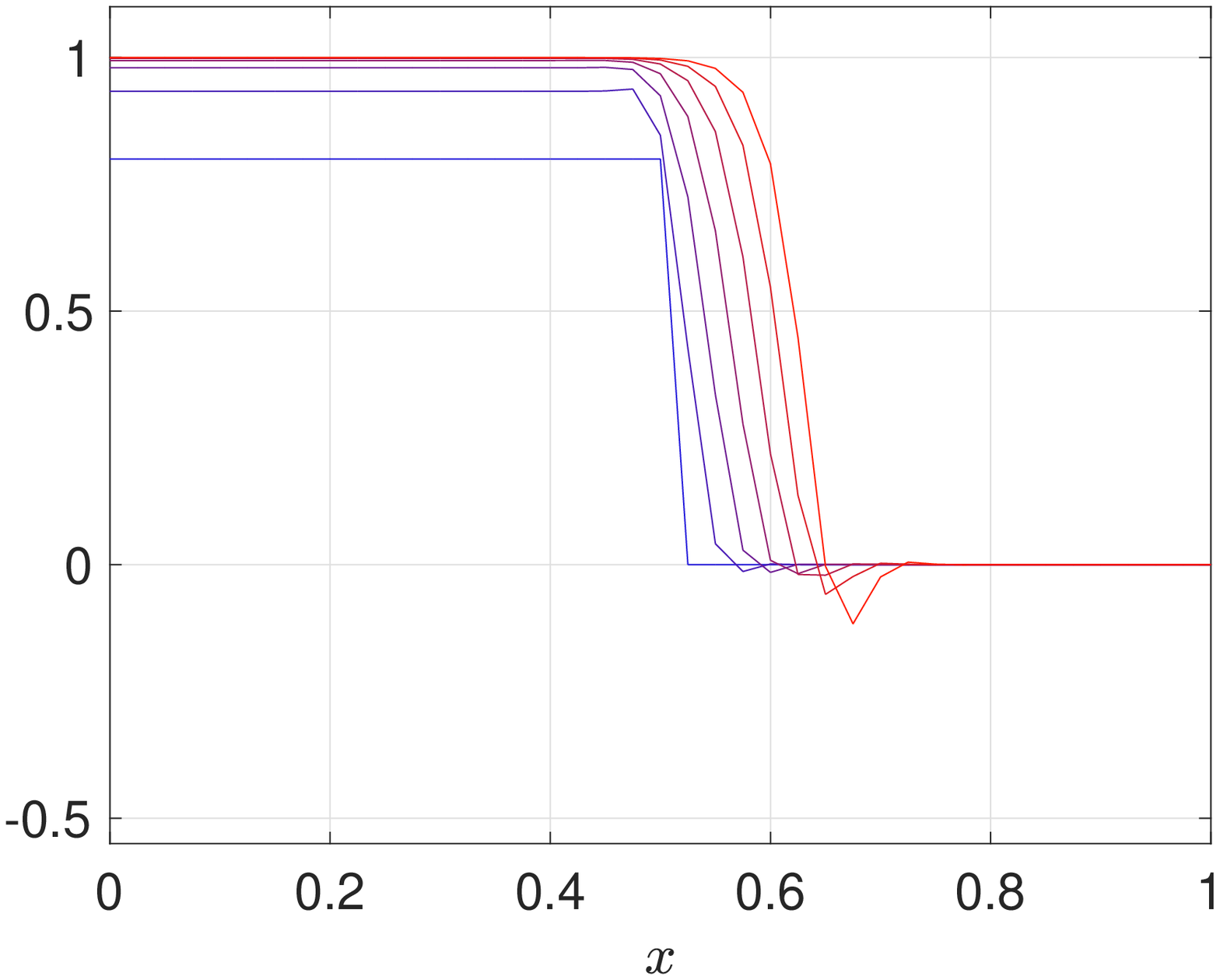}
\caption{Continuous $P_1$ finite-element discretization of problem
\eqref{3.eq}-\eqref{3.bic} in the variable $u$, using $N_{\rm el}=20$ elements 
(left) and $N_{\rm el}=40$ elements (right). The time step size is in both cases
$\triangle t=1/6$, and the solutions move from left to right.} 
\label{fig.u}  
\end{figure}

These results motivate the introduction of the exponential transformation
$u=\exp(\lambda)$. We are choosing the same initial datum as before but choosing
$u_0(x)=10^{-16}$ instead of $u_0(x)=0$ to allow for the exponential transformation. 
Figure \ref{fig.lam} shows the solutions to the continuous $P_1$
finite-element approximation associated to problem \eqref{sc.eq}-\eqref{sc.bc} 
in the variable $\lambda_h^k$. The implicit nonlinear scheme is solved again by
Newton's method with relaxation at each time step. The integrals are solved again
by a Gau\ss{}-Legendre quadrature formula of order 8.
Note that if $p=1$, the integrals are of the type $\int_K e^{ax+b}(cx+d)dx$
and thus can be integrated exactly. The discrete densities
$\exp(\lambda_h^k)$ are positive by construction. However, we need more
relaxation in the Newton method when higher-order schemes $p>1$ are used,
which slows down the algorithm.

\begin{figure}[ht]
\includegraphics[width=0.49\textwidth]{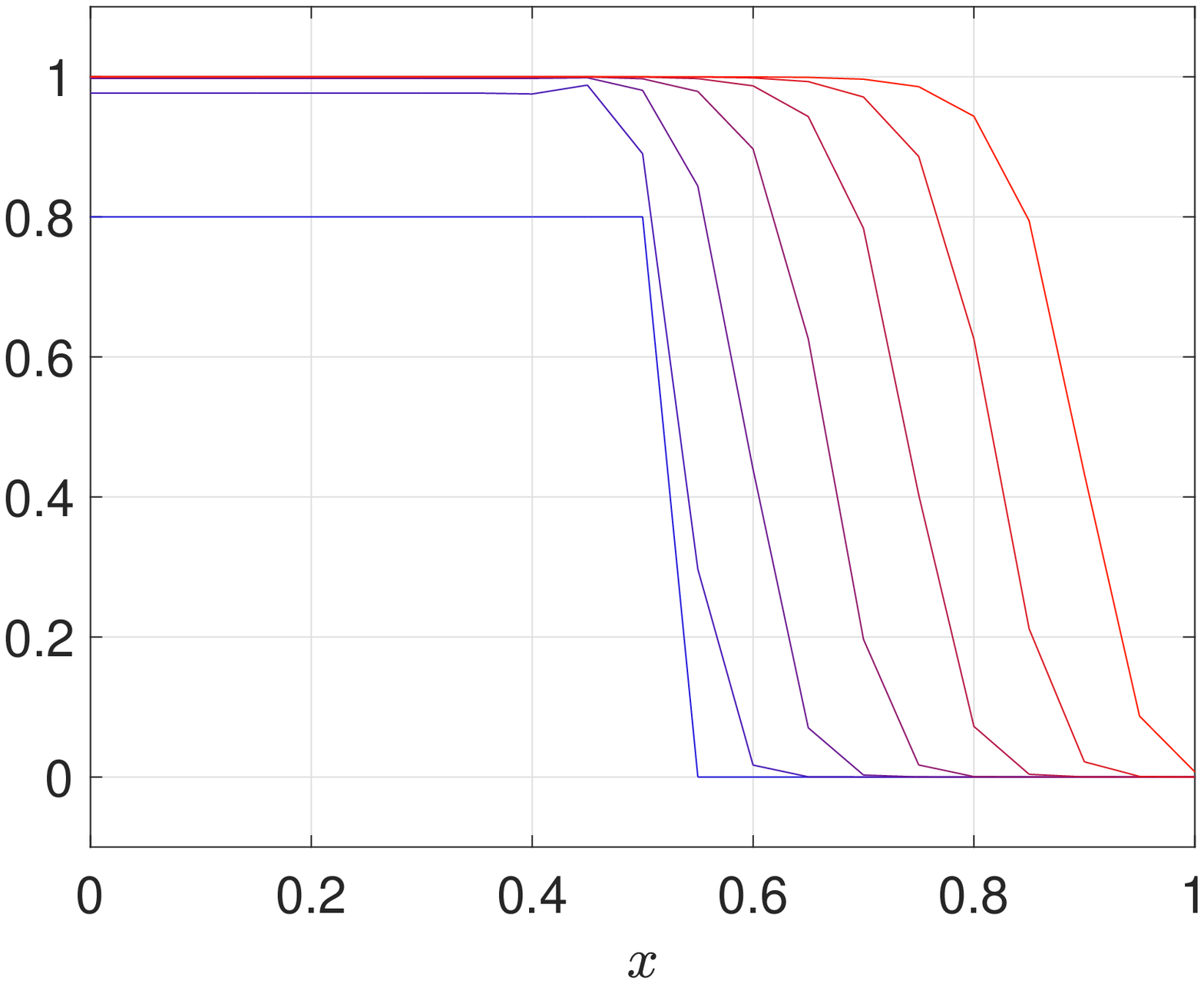}
\includegraphics[width=0.49\textwidth]{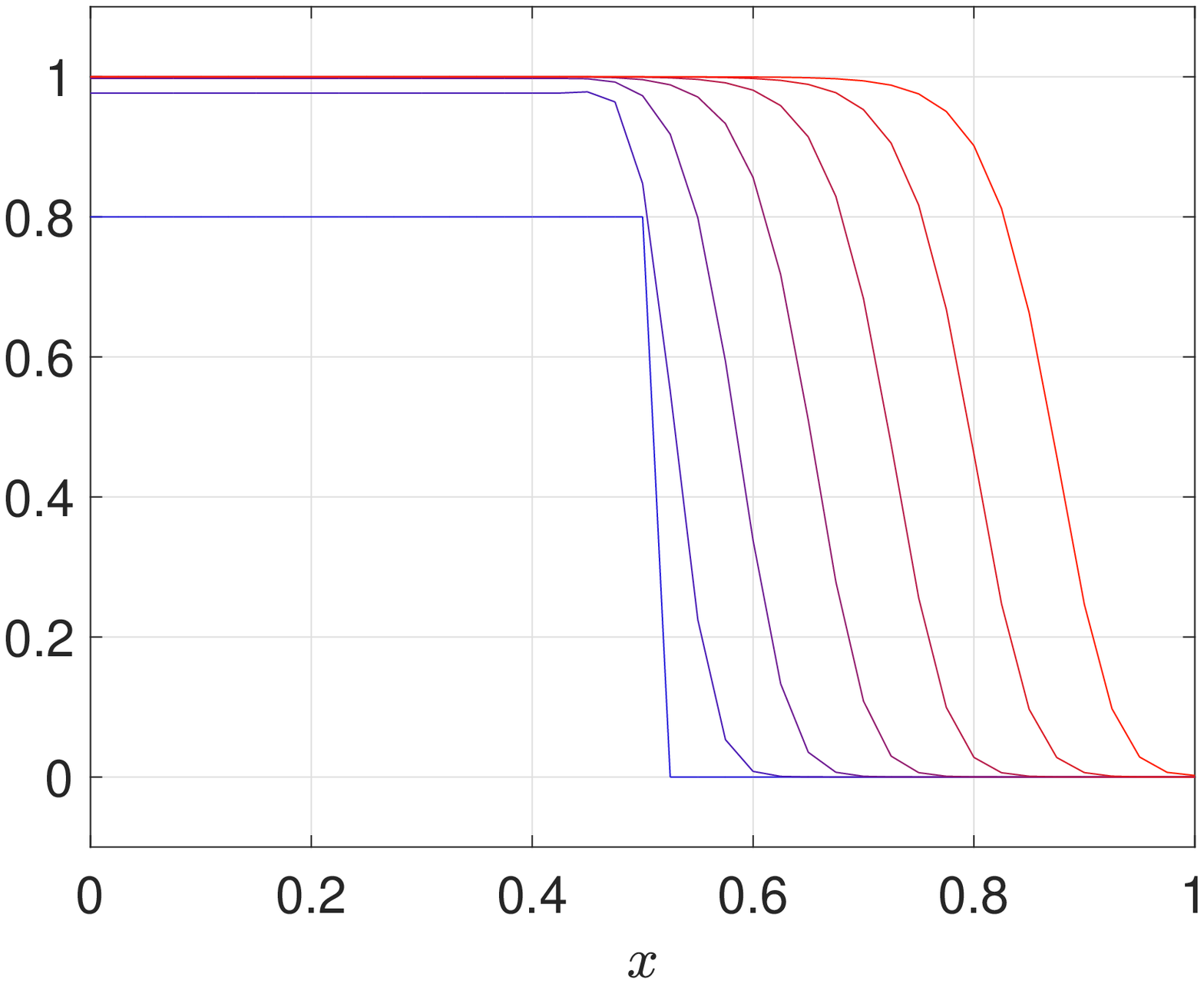}
\caption{Continuous $P_1$ finite-element discretization of problem
\eqref{sc.eq}-\eqref{sc.bc} in the variable $\lambda^k$, using $N_{\rm el}=20$ elements 
(left) and $N_{\rm el}=40$ elements (right). The time step size is in both cases
$\triangle t=1/6$ and the end time is $T=20$.}  
\label{fig.lam}
\end{figure} 

Therefore, we employ a discontinuous Galerkin method with polynomial order
$p\ge 1$ for problem \eqref{sc.eq}-\eqref{sc.bc} in the variable $\lambda_h^k$;
see scheme \eqref{dg}. The regularization term is not necessary for the numerics,
i.e., we set $\eps=0$ in \eqref{dg} for our simulations. 
Figure \ref{fig.dg.u08} illustrates the discrete solutions
for polynomial orders $p=1,2,3$, indicating that the method is stable with respect
to the order. The jumps are due to the discontinuous Galerkin method.

\begin{figure}[ht]
\includegraphics[width=0.49\textwidth]{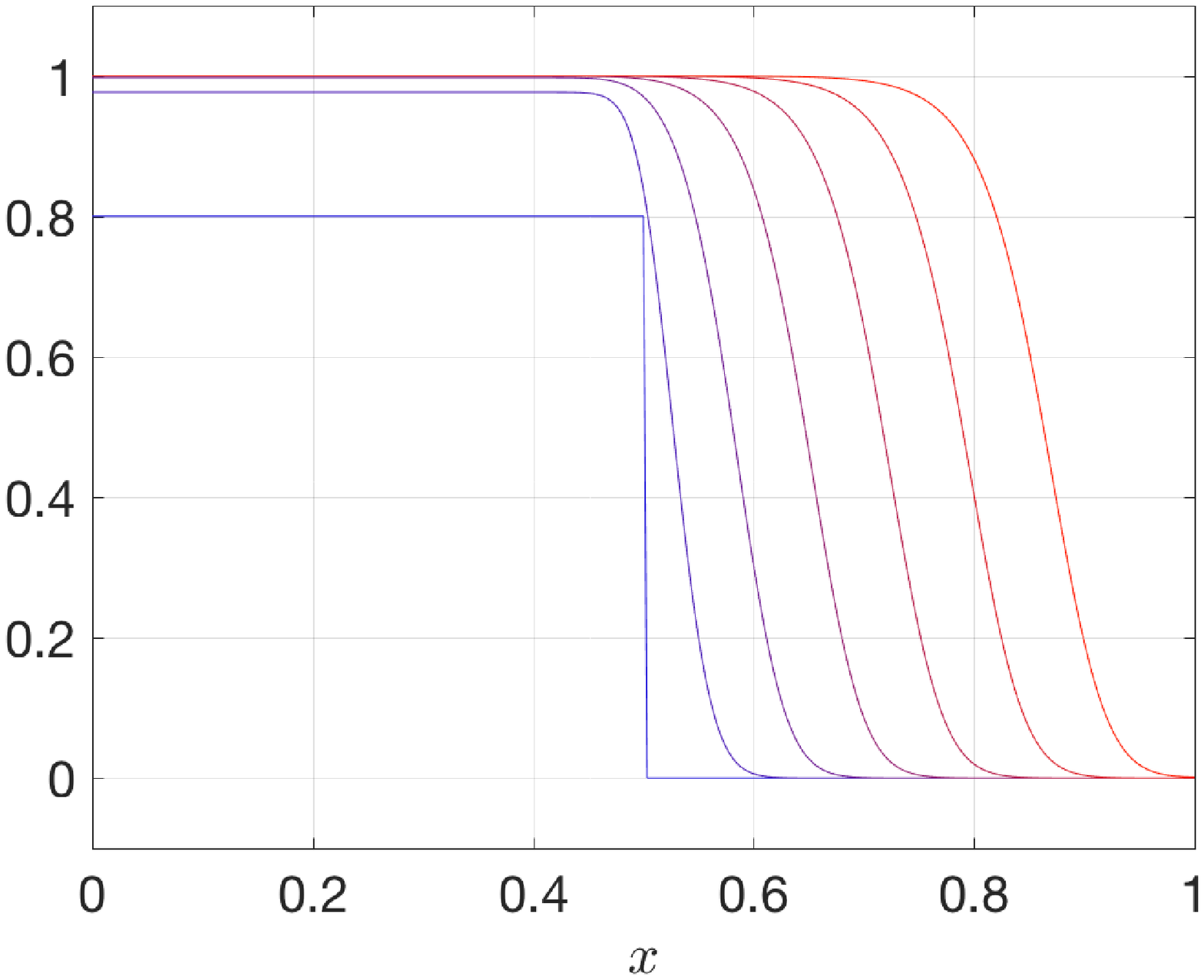}  
\includegraphics[width=0.49\textwidth]{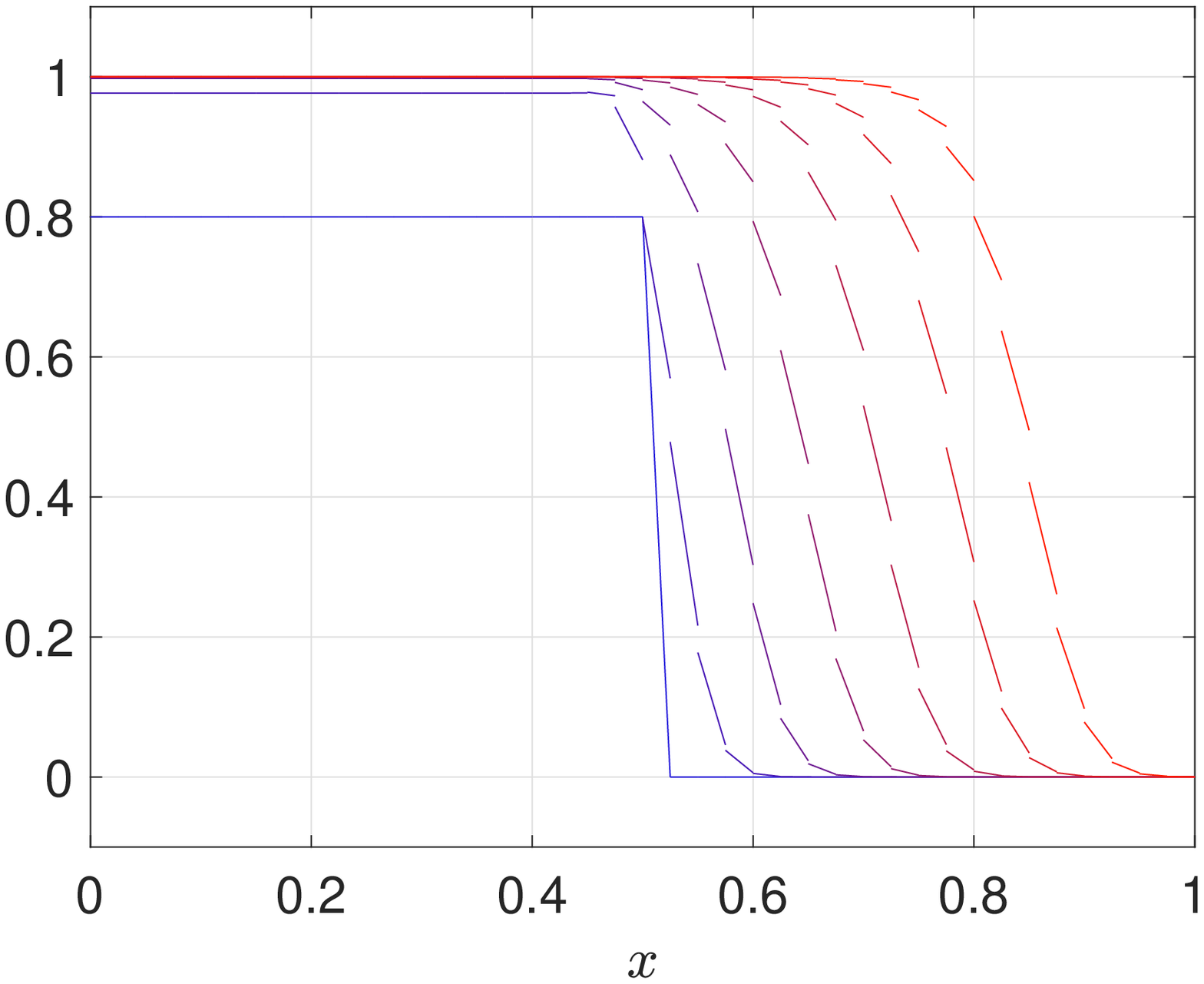}
\includegraphics[width=0.49\textwidth]{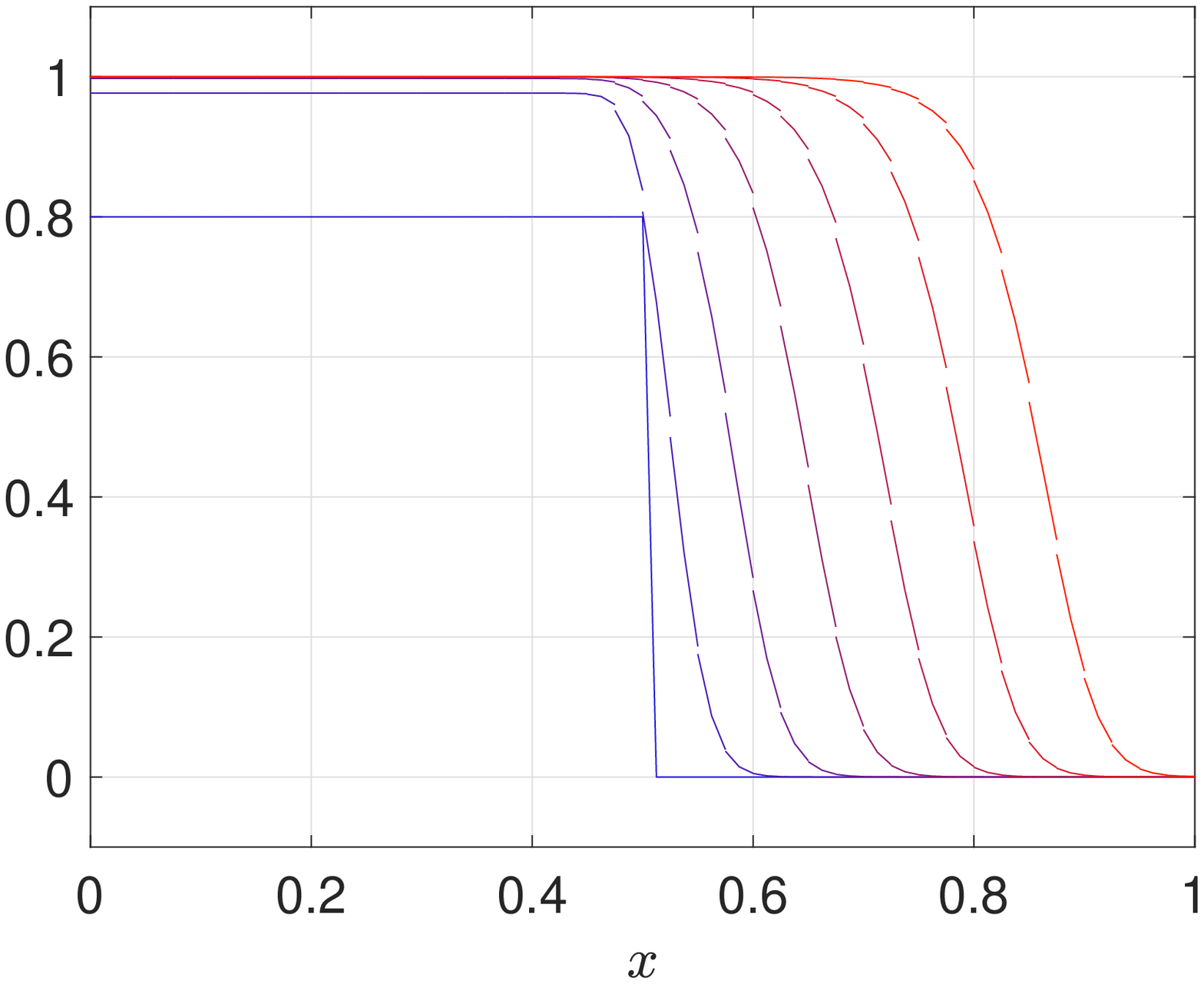}
\includegraphics[width=0.49\textwidth]{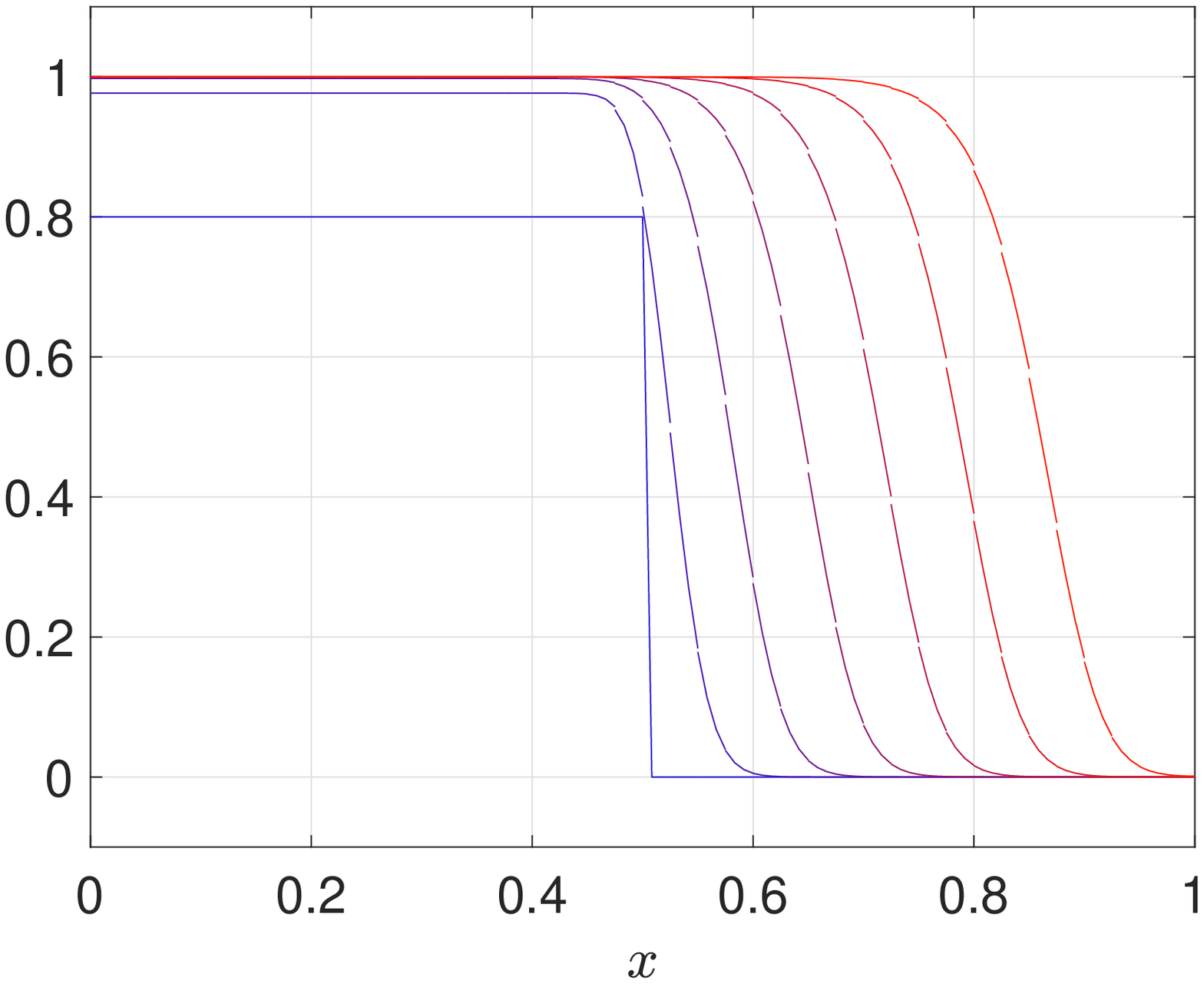}
\caption{Reference solution computed from the $P_1$ finite-element scheme
with $N_{\rm el}=300$ elements and $\triangle t=1/3$
(left top) and solutions computed from the DG scheme with $N_{\rm el}=40$ elements,
$\triangle t=1/3$, and polynomial order $p=1$ (right top), $p=2$ (left bottom),
and $p=3$ (right bottom). The initial datum is $u_0(x)=0.8$ for $0<x<1/2$ and
$u_0(x)=10^{-16}$ else.}
\label{fig.dg.u08}
\end{figure}

Figure \ref{fig.dg.u1} represents the discrete solutions with the same numerical
parameters as in Figure \ref{fig.dg.u08} but with the initial datum
$u_0(x)=1$ for $0<x<1/2$ and $u_0(x)=0$ else. Also in this example, the
lower and upper bounds $0\le \exp(\lambda_h^k)\le 1$ are always satisfied.

\begin{figure}[ht]
\includegraphics[width=0.49\textwidth]{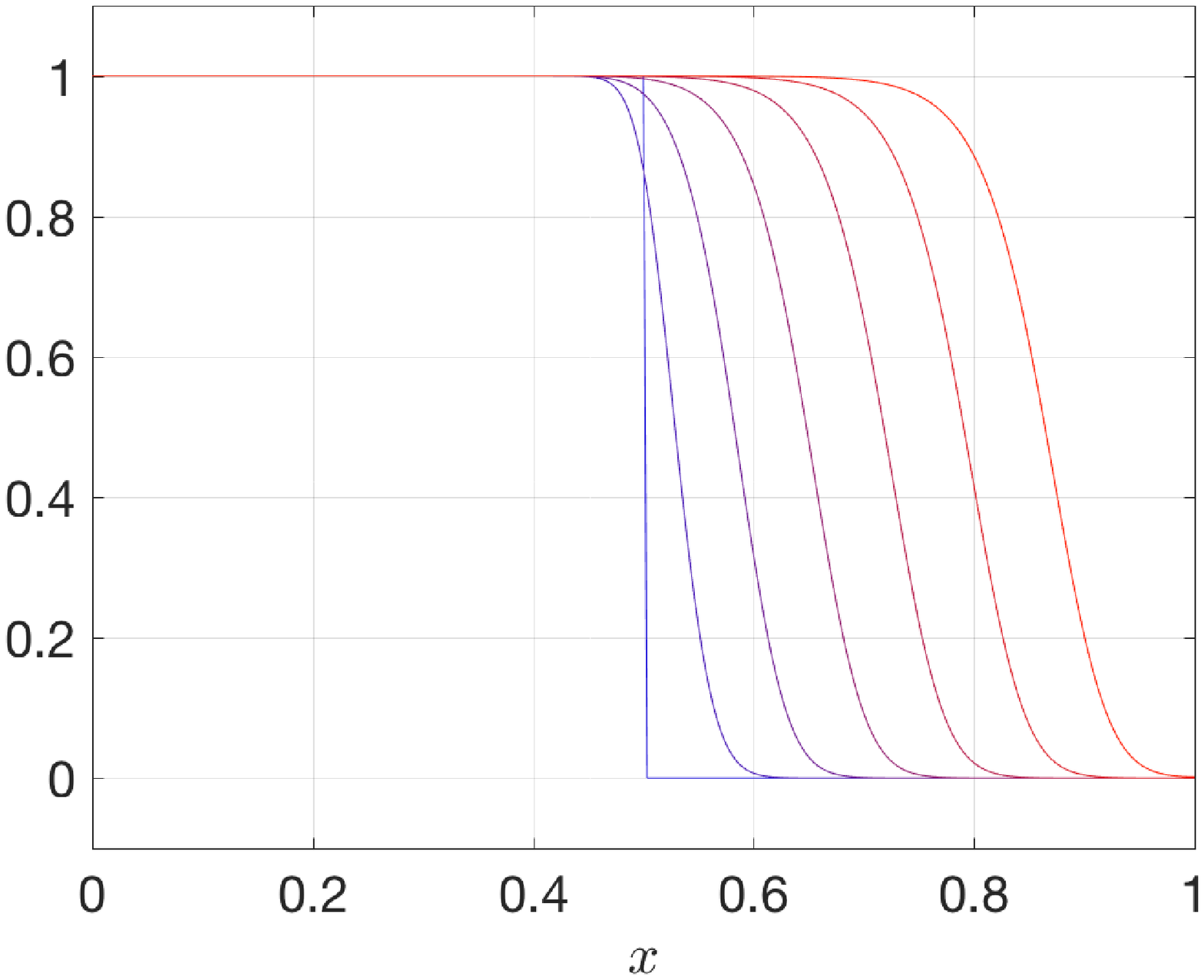}
\includegraphics[width=0.49\textwidth]{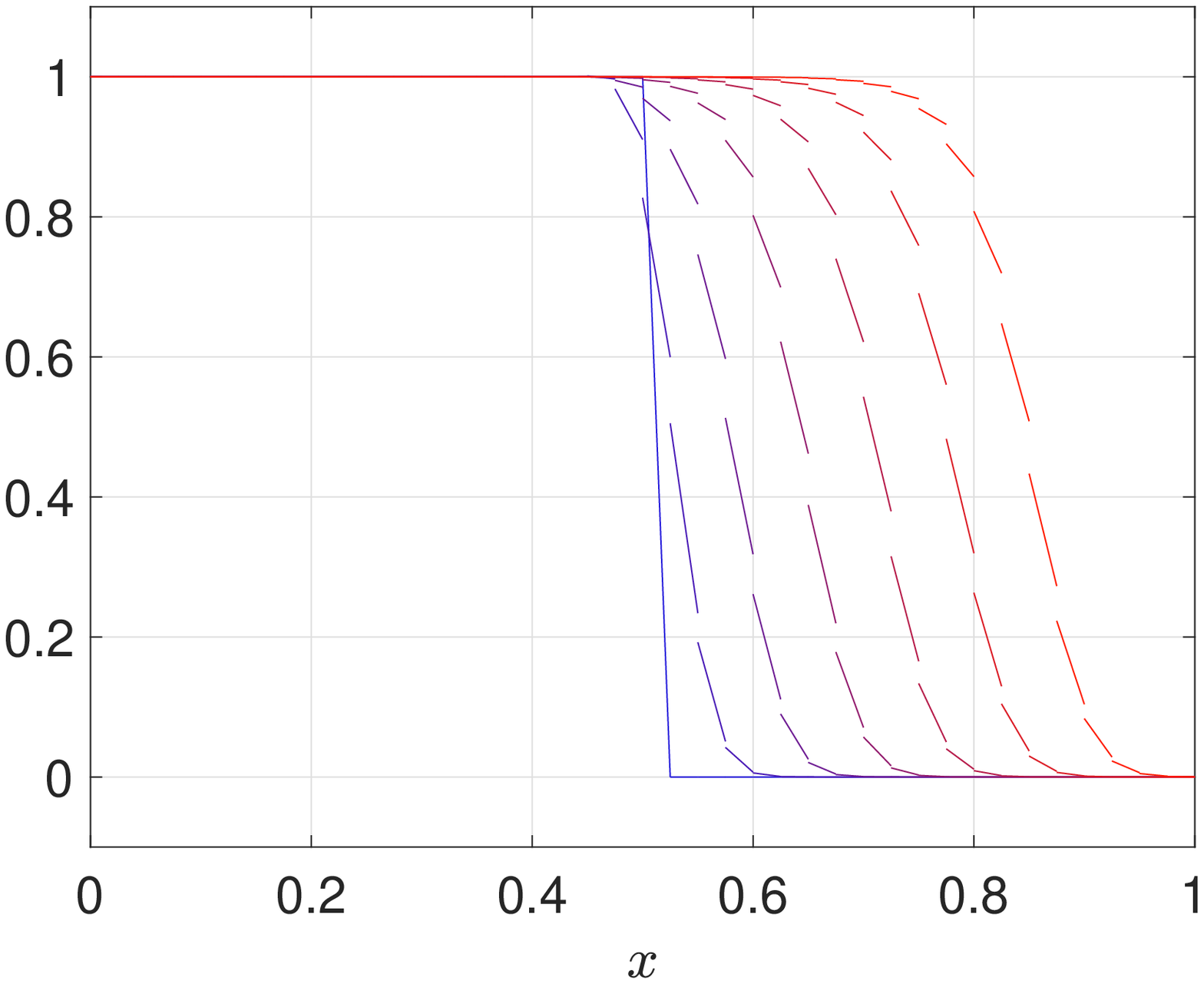}
\includegraphics[width=0.49\textwidth]{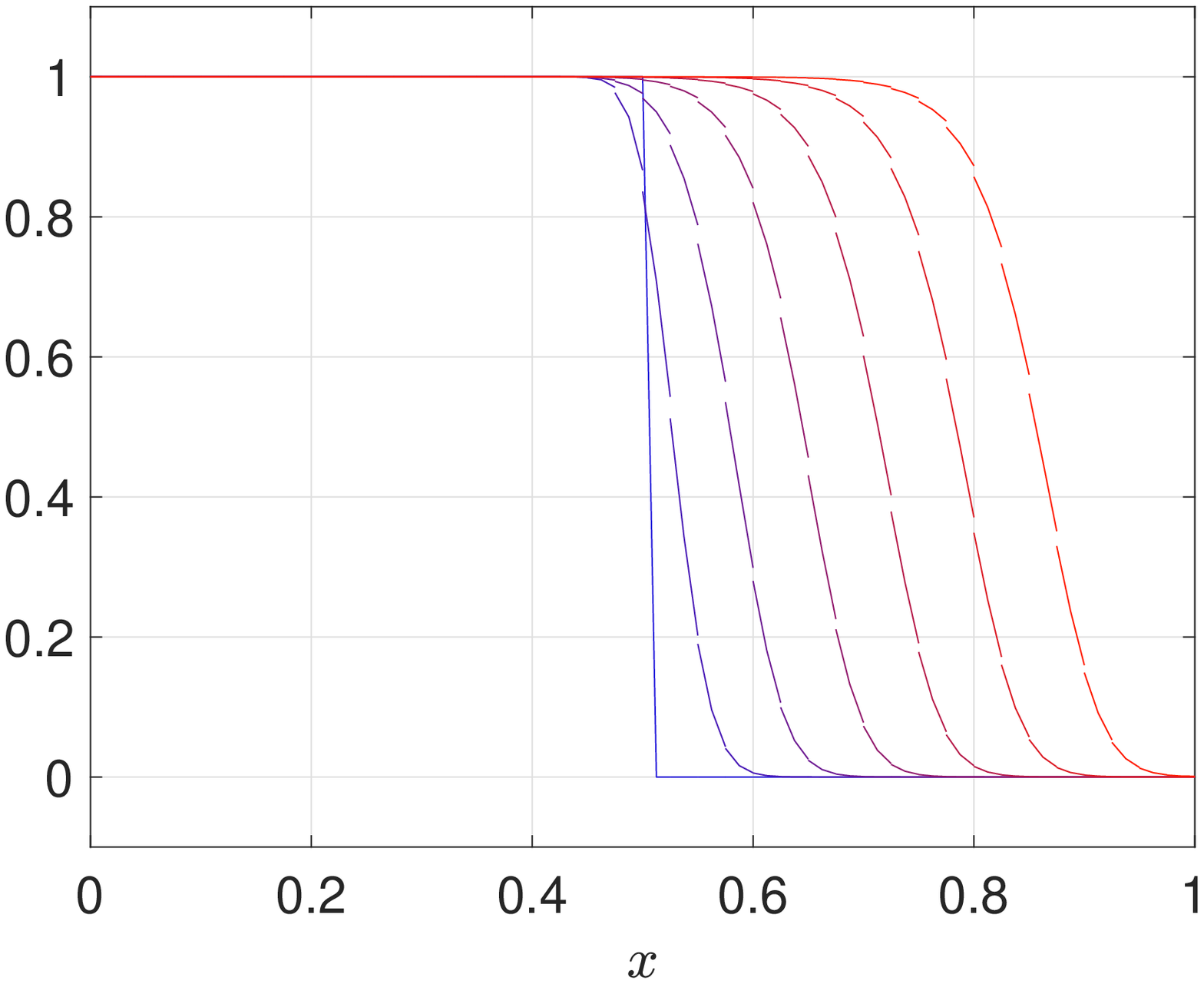}
\includegraphics[width=0.49\textwidth]{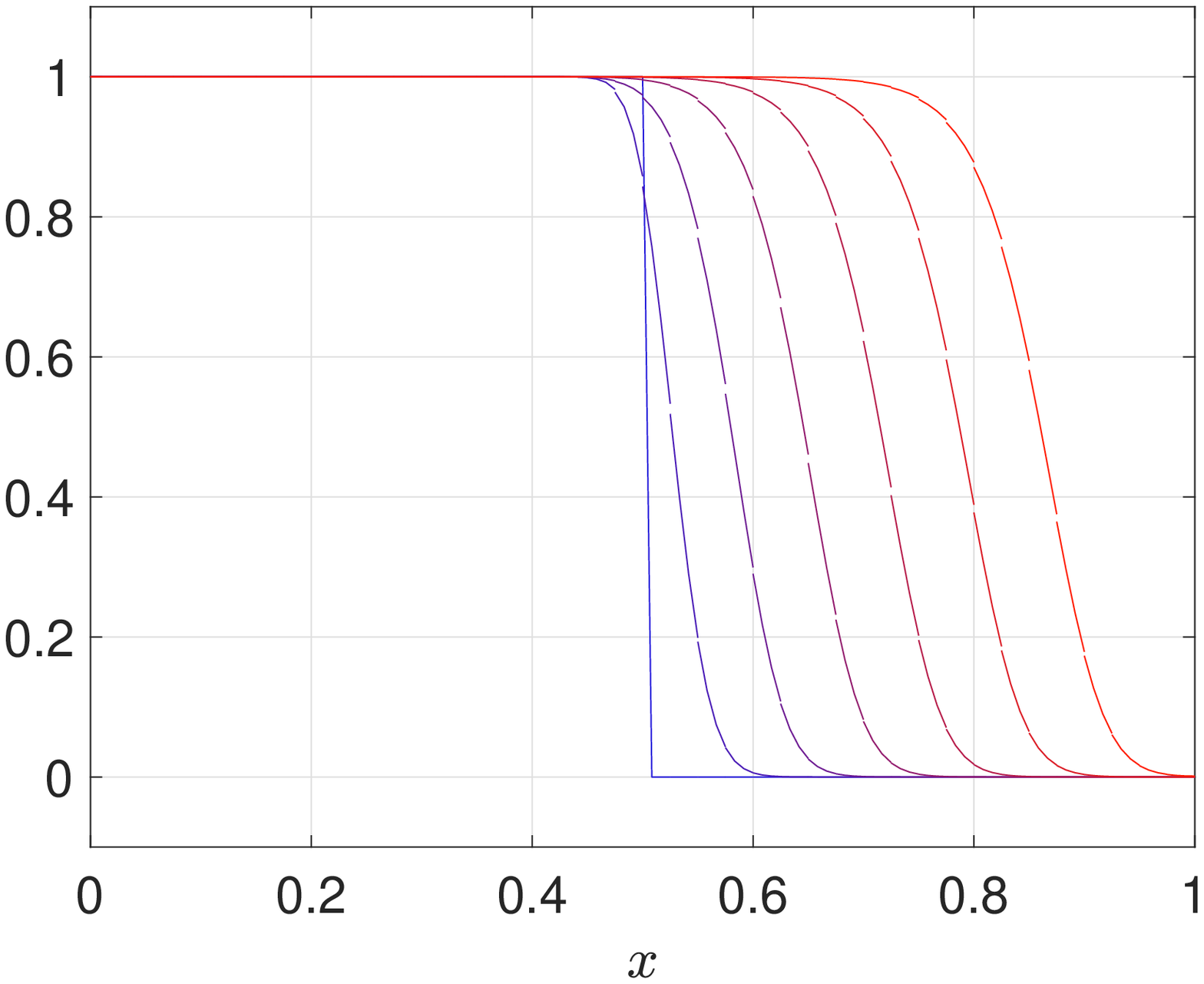}
\caption{Reference solution computed from the $P_1$ finite-element scheme
with $N_{\rm el}=300$ elements and $\triangle t=1/3$
(left top) and solutions computed from the DG scheme with $N_{\rm el}=40$ elements,
$\triangle t=1/3$, and polynomial order $p=1$ (right top), $p=2$ (left bottom),
and $p=3$ (right bottom). The initial datum is $u_0(x)=1$ for $0<x<1/2$ and
$u_0(x)=10^{-16}$ else.}
\label{fig.dg.u1}    
\end{figure}
   

\subsection{Entropy decay}

Proposition \ref{prop.decay} shows that the discrete entropy $S_h^k$ decays
exponentially fast if $S_h^0<|\Omega|=1$. To illustrate this behavior numerically,
we consider the one-group model with initial condition $u_0(x)=1$ for $0<x<1/2$
and $u_0(x)=10^{-16}$ else. 
Figure \ref{fig.ent} (left) shows that there are two different
slopes. For small times, the entropy decay is rather slow. When the time step
$k$ is sufficiently large such that $\exp(\lambda_h^k)>\eps$
for some $\eps>0$, the reaction dominates, and the entropy decay becomes faster. 
This behavior becomes even more apparent in the case of pure diffusion (i.e.\
without reaction terms), illustrated in Figure \ref{fig.ent} (right).
We remark that in this situation, the total mass is conserved numerically.

\begin{figure}[ht]
\includegraphics[width=0.49\textwidth]{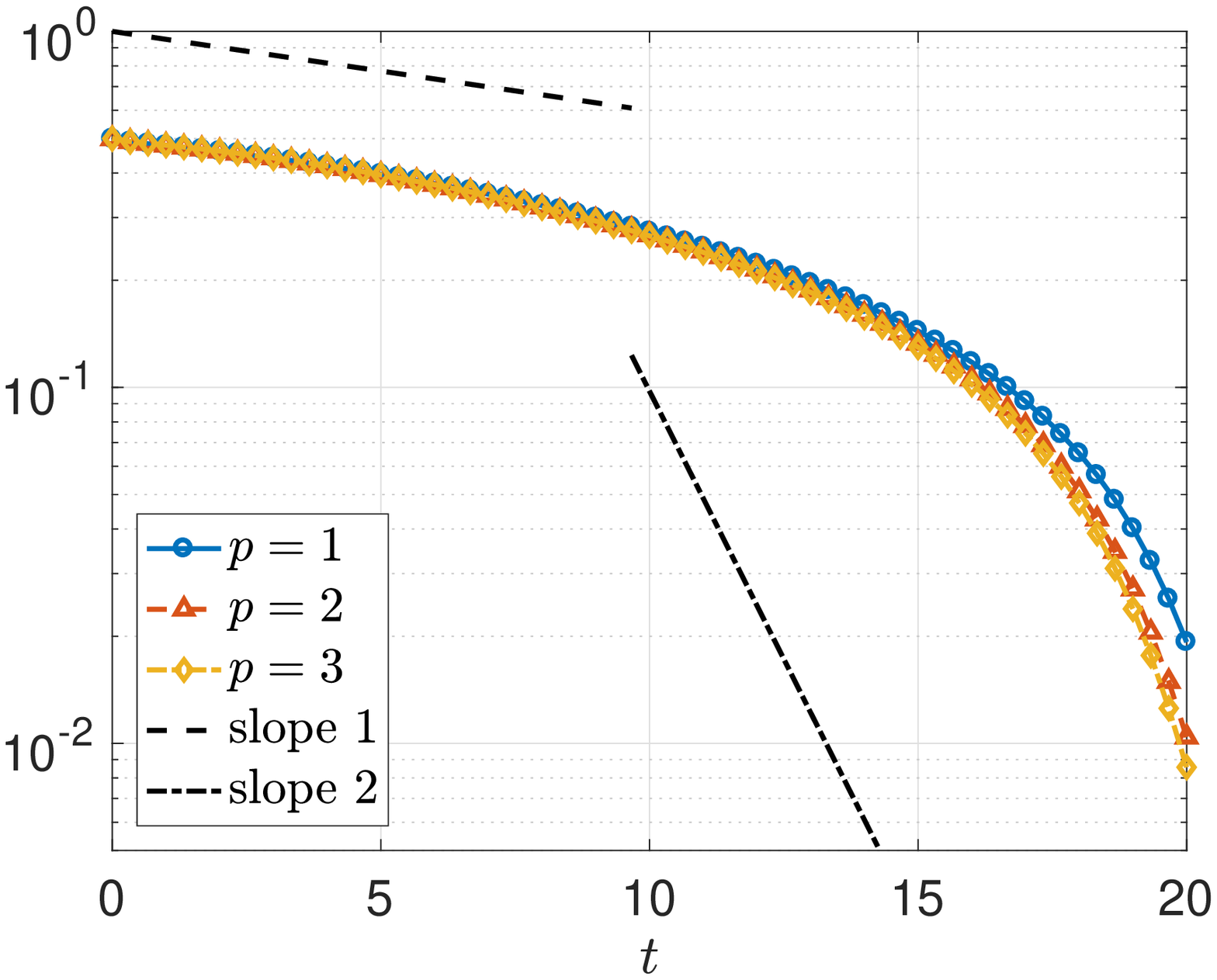}
\includegraphics[width=0.49\textwidth]{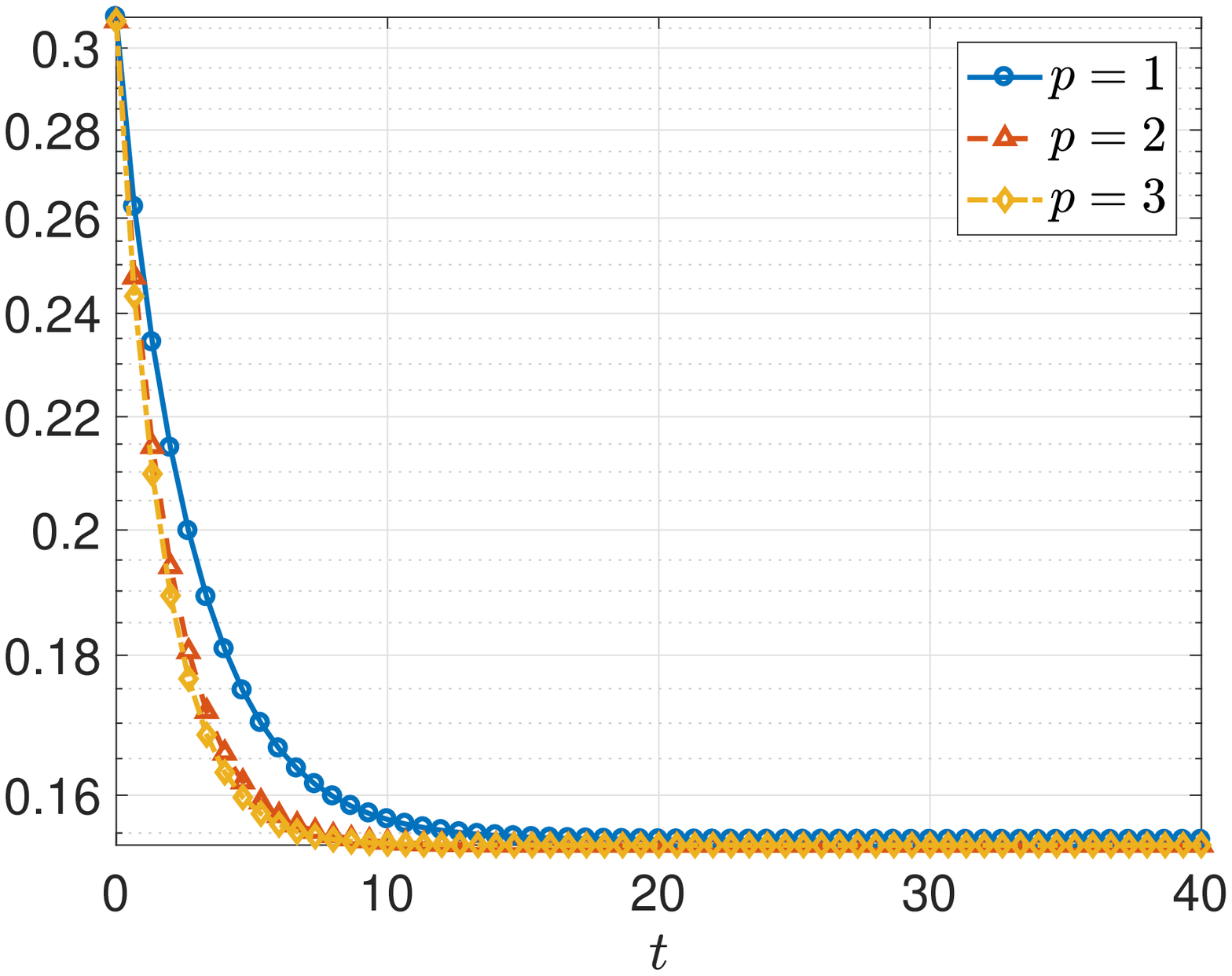}
\caption{Left: Entropy decay for the one-group model. The reference slopes 
are $t\mapsto 0.95^t$ (slope 1) and $t\mapsto 0.5^t$ (slope 2). 
Right: entropy decay for the pure diffusion equation. Both figures are in
semi-log scale.}
\label{fig.ent}
\end{figure} 

Figure \ref{fig.ent2} shows the entropy decay in semi-log scale for the
initial data $u_0(x)=n$ for $0<x<1/n$ and $u_0(x)=10^{-16}$ else for $n=3,6,12$.
Then
$$
	S_h^0 = \int_{0}^{\frac{1}{n}}(n\log(n)-n+1) dx +\int_{\frac{1}{n}}^1 1dx =\log n
$$
is larger than $|\Omega|=1$ if $n>e$.
We observe a region in which the decay rate is very small and it becomes
smaller when $n$ (and $S_h^0$) increases. This indicates that the assumption
$S_0^h<|\Omega|$ is not just technical to derive exponential time decay.

\begin{figure}[ht]
\includegraphics[width=0.49\textwidth]{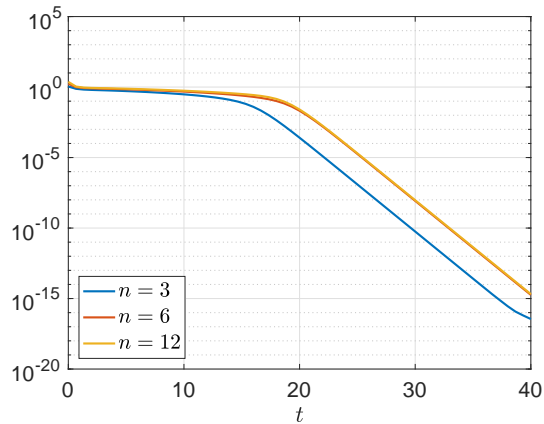}
\caption{Entropy decay for the one-group model with the initial datum
$u_0(x)=n$ for $0<x<1/n$ and $u_0(x)=10^{-16}$ else.}
\label{fig.ent2}
\end{figure} 


\subsection{Traveling waves}

We are looking for traveling-wave solutions to \eqref{3.eq} with $D=1$.
Setting $u(x,t)=\phi(s)$ with
$s=x-ct$, the Fisher-KPP equation can be rewritten as a system of first-order
differential equations:
\begin{equation}\label{3.tw}
  \phi' = -c\phi + \psi(\psi-1), \quad \psi' = \phi, \quad t>0.
\end{equation}
We choose the initial data $\phi(0)=1$ and $\psi(0)=-10^{-10}$. The 
(reference) traveling-wave
solution $\phi(s)$, computed from \eqref{3.tw} using the Matlab command
{\tt ode45}, is compared in Figure \ref{fig.tw} with the DG solution 
$\exp(\lambda_h^k)$, computed from the DG scheme \eqref{dg}, and the 
continuous $P_1$ finite-element solution $u_h$. The solutions are shown
at the time instances $t=0$, $t=T/6$, $t=T/3$, and $t=T/2$ (with $T=20$).
Both approximations are
diffusive, i.e., the traveling-wave speed is overestimated by the DG and
finite-element solutions. On the finer mesh with $N_{\rm el}=80$ elements,
the DG solution is clearly less diffusive compared to the other discrete solutions
on the coarser mesh with $N_{\rm el}=50$ elements. Better approximations are
expected by using a higher-order time discretization. Structure preservation
of higher-order temporal approximations is a delicate topic (see, e.g., 
\cite{JuSc17}) and will be studied in a future work. Also an error analysis 
is postponed to a future work.

\begin{figure}[ht]
\includegraphics[width=0.49\textwidth]{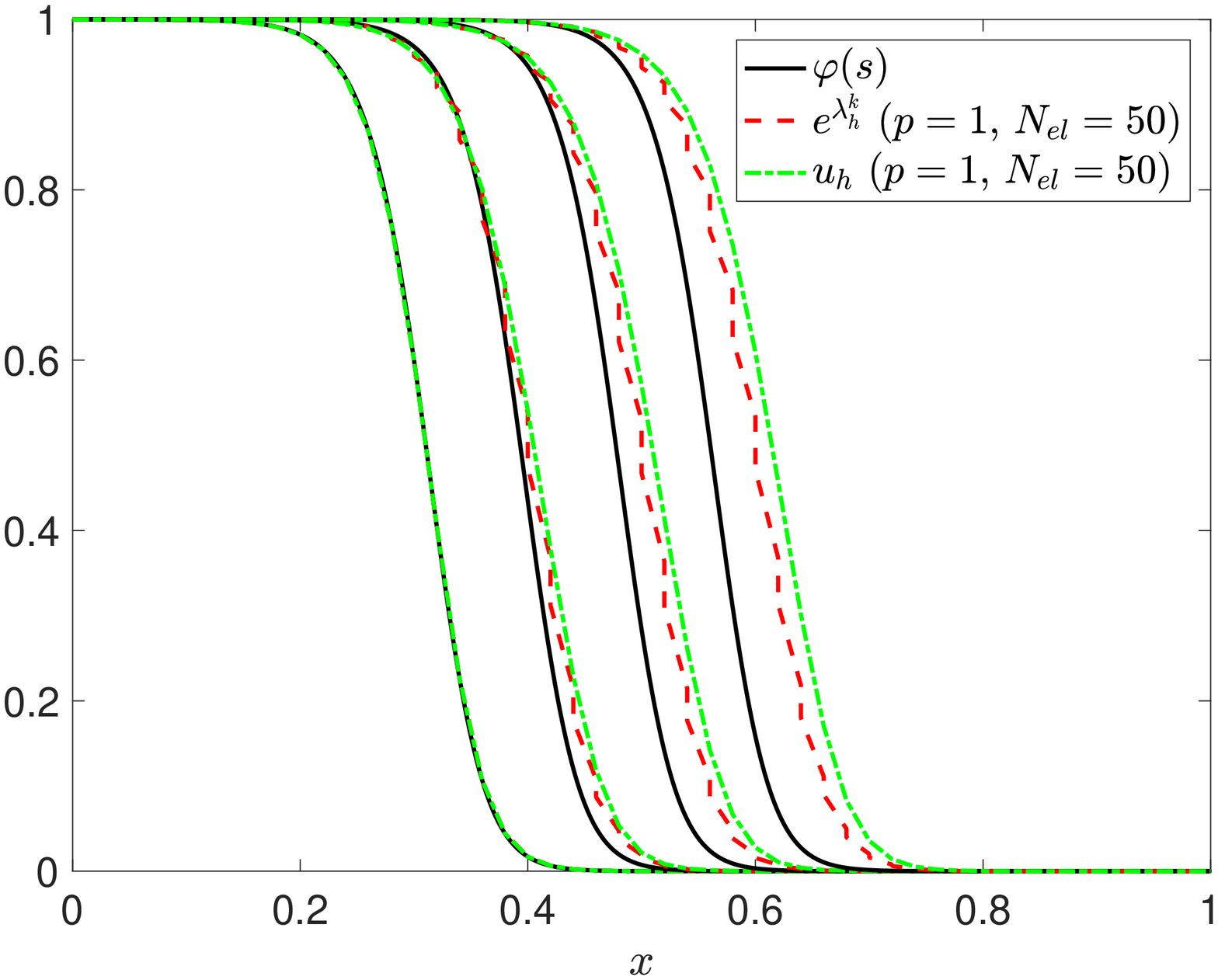}
\includegraphics[width=0.49\textwidth]{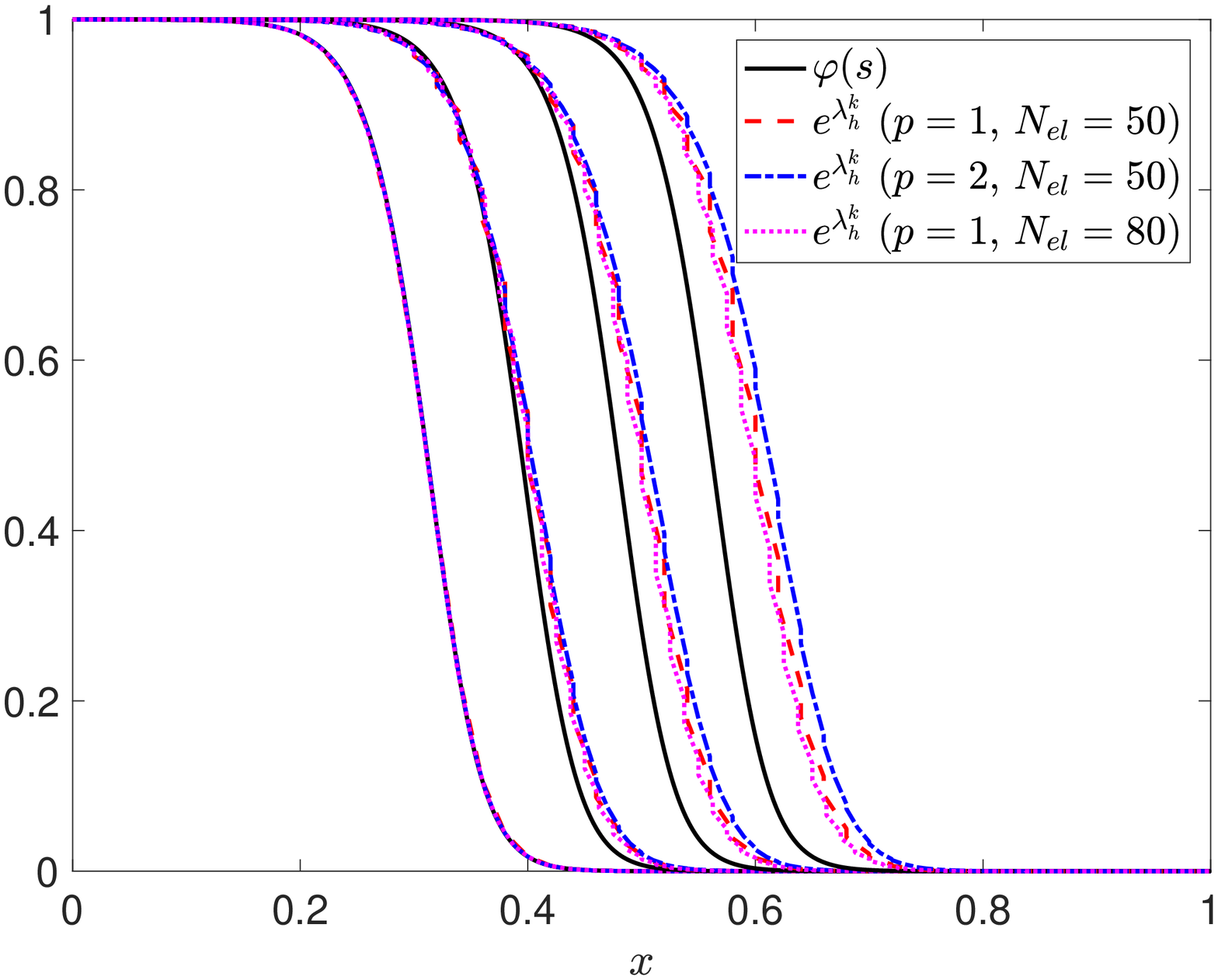}
\caption{Comparison of the traveling-wave solution $\phi(s)$, the DG solution
$\exp(\lambda_h^k)$ (with $N_{\rm el}=50$ or $N_{\rm el}=80$ and $\triangle t=1/3$), 
and the finite-element solution $u_h$. }
\label{fig.tw}
\end{figure} 


\begin{appendix}
\section{Linear elements: Existence of solutions for $\eps=0$}\label{app}

We show that the regularization term $\eps\int_\Omega \lambda_h^k\phi_h dx$
in the DG scheme \eqref{dg} is not needed if we consider linear elements.
 
\begin{proposition}[Existence for $p=1$]\label{prop.p1}
Let $p=1$, $\lambda_h^{k-1}\in V_h$, and $S_0^h<|\Omega|$. 
Then there exists a solution $\lambda_h^k\in V_h$ to \eqref{dg} with $\eps=0$.
\end{proposition}

\begin{proof}
The proof is based on the idea of the proof of \cite[Lemma 3.10]{ACC18}.
Let $\eps>0$ and let $\lambda_h^k\in V_h$ be a solution to \eqref{dg} given by
Proposition \ref{prop.ex}. In order to emphasize the dependency on $\eps$,
we write $\lambda_\eps:=\lambda_h^k$. Our goal is to derive an $\eps$-uniform
$L^\infty(\Omega)$ bound for $\lambda_\eps$. 

{\em Step 1:} We derive first some estimates for $e^{\lambda_\eps}$.
Lemma \ref{lem.mass} shows that
\begin{equation}\label{a.est1}
  \delta \le \int_\Omega e^{\lambda_\eps}dx\le M,
\end{equation}
where $\delta=\sigma_-(S_h^0/|\Omega|)$ and $M=\sigma_+(S_h^0/|\Omega|)$.
Since we assumed that $S_h^0<|\Omega|$, we have $\delta>0$. 
Using the coercivity estimate \eqref{2.aux2}, the inequality \eqref{2.epiB},
and $e^{\lambda_\eps}(e^{\lambda_\eps}-1)\lambda_\eps\ge 0$, we find that
$$
  S_h^k + \frac{\triangle t}{2}\sum_{K\in\T_h}\int_K e^{\lambda_\eps}
	|\na\lambda_\eps|^2 dx + \frac{\triangle t}{3}\sum_{f\in\E_h}\int_f
	\frac{p^2}{{\tt h}}\alpha(\lambda_\eps)|\jump{\lambda_\eps}|^2 ds
	\le S^{k-1}_h \le S_h^0,
$$
where we used in the last step the monotonicity of $k\mapsto S_h^k$, 
guaranteed by Lemma \ref{lem.epi}. Consequently,
\begin{equation}\label{a.est2}
  \sum_{K\in\T_h}\int_K e^{\lambda_\eps}|\na\lambda_\eps|^2 dx
	+ \sum_{f\in\E_h}\int_f\frac{p^2}{{\tt h}}\alpha(\lambda_\eps)
	|\jump{\lambda_\eps}|^2 ds \le M':= \frac{3S_h^0}{\triangle t}.
\end{equation}

{\em Step 2:} We claim that for any $\eps>0$, there exists an element $K_\eps\in\T_h$
and a constant $\mu>0$, independent of $\eps$, such that
\begin{equation}\label{a.step2}
  \|\lambda_\eps\|_{L^\infty(K_\eps)} \le \mu.
\end{equation}
For the proof, let $N\in\N$ be the number of elements in $\T_h$. The lower and
upper bounds \eqref{a.est1} imply the existence of an element $K_\eps\in\T_h$
such that
\begin{equation}\label{a.aux1}
  \frac{\delta}{N}\le \int_{K_\eps} e^{\lambda_\eps}dx \le M.
\end{equation}
By the mean-value theorem, there exists $x_\eps\in K_\eps$ such that
$$
  e^{\lambda_\eps(x_\eps)} = \frac{1}{|K_\eps|}\int_{K_\eps}e^{\lambda_\eps(x)}dx.
$$
Then \eqref{a.aux1} gives
$$
  |\lambda_\eps(x_\eps)| = \bigg|\log\bigg(\frac{1}{|K_\eps|}\int_{K_\eps}
	e^{\lambda_\eps(x)}dx\bigg)\bigg|
	\le \max\bigg\{\bigg|\log\frac{M}{|K_\eps|}\bigg|,\log
	\frac{\delta}{N|K_\eps|}\bigg|\bigg\}.
$$
Since $\lambda_\eps$ is a polynomial of degree one on $K_\eps$, by assumption,
its gradient is constant on $K_\eps$, and we deduce from \eqref{a.est2} and
\eqref{a.aux1} that
$$
  |\na\lambda_\eps|^2 = \frac{\int_{K_\eps}e^{\lambda_\eps}|\na\lambda_\eps|^2dx}{
	\int_{K_\eps}e^{\lambda_\eps}dx}
	\le \frac{M'N}{\delta} \quad\mbox{on }K_\eps.
$$
Combining the last two estimates, it follows for $x\in K_\eps$ that
\begin{align*}
  |\lambda_\eps(x)|&\le |\lambda_\eps(x_\eps)| 
	+ |x-x_\eps|\int_0^1|\na\lambda_\eps(x\eps+\theta(x-x_\eps)|d\theta \\
  &\le \max\bigg\{\bigg|\log\frac{M}{|K_\eps|}\bigg|,\log
	\frac{\delta}{N|K_\eps|}\bigg|\bigg\} + \frac{M'N}{\delta} =: \mu,
\end{align*}
which shows the claim. 

{\em Step 3:} We wish to prove a uniform $L^\infty$ bound for $\lambda_\eps$ 
on the faces or edges of $K_\eps$. Let $\mu>0$ and $K_\eps\in\T_h$ such that
$\|\lambda_\eps\|_{L^\infty(K_\eps)}\le\mu$. Set $K_-:=K_\eps$ and consider
neighboring elements $K_+\in\T_h$ satisfying $f\in \pa K_-\cap\pa K_+\neq\emptyset$. 
Furthermore, let $\lambda_\pm=\lambda_\eps|_{K_\pm}$. We claim that there 
exists $C_\mu>0$, independent of $\eps$, such that
\begin{equation}\label{a.step3}
  \|\lambda_+\|_{L^\infty(f)}\le C_\mu.
\end{equation}
The idea is to prove an $L^2(f)$ estimate for $\lambda_\eps$, as the equivalence
of all norms in the finite-dimensional setting then implies the desired
$L^\infty(f)$ bound.

Observe that
$$
  \max\big\{(e^{\lambda_\eps})_-,(e^{\lambda_\eps})_+\big\}
	\ge (e^{\lambda_\eps})_- \ge \exp(-\|\lambda_-\|_{L^\infty(f)})
	\ge \exp(-\|\lambda_-\|_{L^\infty(K_-)}) \ge \exp(-\mu).
$$
Then we can estimate the stabilization function $\alpha$ according to
$$
  \alpha(\lambda_\eps) \ge \frac32 C_{\rm inv}^2 
	\exp(-\|\lambda_-\|_{L^\infty(K_-)})^2\exp(\|\lambda_-\|_{L^\infty(K_-)}
	\ge \frac32 C_{\rm inv}^2 e^{-\mu}.
$$
To estimate the $L^2(f)$ norm of $\lambda_-|_f$, we use the inequality
$|\lambda_+|\le |\lambda_+ - \lambda_-| + |\lambda_-| = |\jump{\lambda_\eps}| 
+ |\lambda_\eps|$ on $f$. Then \eqref{a.step2} yields
$$
  \int_f|\lambda_+|^2 ds \le 2\int_f\big(|\lambda_\eps|^2 + |\jump{\lambda_\eps}|^2
	\big)ds 
	\le 2|f|\mu^2 + \frac{4e^{\mu}}{3C_{\rm inv}^2}\int_f\alpha(\lambda_\eps)
	|\jump{\lambda_\eps}|^2 ds =: \beta_\mu,
$$
and $\beta_\mu$ is uniform in $\eps$ (but not in $h$) in view of \eqref{a.est2}.
We conclude that
$$
  \|\lambda_\eps\|_{L^\infty(f)} \le C\|\lambda_\eps\|_{L^2(f)}
	\le C\beta_\mu^{1/2}.
$$

{\em Step 4:} Denote by $(\phi_0,\ldots,\phi_d)$ the basis of $P_1(K_+)$ such that
$\phi_i(e_j)=\delta_{ij}$ for $i$, $j=0,\ldots,d$, where the vertices $e_i$
of $K_+$ are ordered in such a way that $e_0\not\in f$. Then we can formulate
$\lambda_+$ on $K_+$ as
$$
  \lambda_+(x) = \sum_{i=0}^d a_i^\eps\phi_i(x), \quad x\in K_+,
$$
where $a_i^\eps=\lambda_+(e_i)$. Estimate \eqref{a.step3} shows that
$\lambda_+(e_i)$ is uniformly bounded at the vertices $a_1,\ldots,a_d$ of $K_+$, i.e.\
$|a_i^\eps|\le C_\mu$ for all $i=1,\ldots,d$.

{\em Step 5:} We wish to estimate the remaining vertex $e_0^\eps$ that is not
an element of $K_-=K_\eps$.
We claim that there exist constants $L_\mu\le U_\mu$, being independent of $\eps$, 
such that
$$
  L_\mu \le a_0^\eps \le U_\mu.
$$
We first prove the upper bound. Using the bound for $a_i^\eps$ for $i=1,\ldots,d$,
we have
$$
  \lambda_+ \ge a_0^\eps\phi_0 - \sum_{i=1}^d|a_i^\eps||\phi_i|
	\ge a_0^\eps\phi_0 - C_\mu\sum_{i=1}^d|\phi_i| \ge a_0^\eps\phi_0 - C_\mu d.
$$
If $a_0^\eps\le 0$, there is nothing to show. Otherwise, it follows from
\eqref{a.aux1} that
$$
  M\ge \int_{K_+} \exp(\lambda_+)dx 
	\ge \int_{K_+}\exp(a_0^\eps\phi_0 - C_\mu d)dx 
	\ge e^{-C_\mu d}\int_{K_+}a_0^\eps\phi_0 dx.
$$
Then, setting $c_0=\int_{K_+}\phi_0dx$, we infer that
$a_0^\eps\le Me^{C_\mu d}/c_0 =: U_\mu$.

The proof of the lower bound is more involved. Let $f_0$ be the face or edge
that is opposite of the vertex $e_0$, and let $f_1,\ldots,f_d$ be the remaining
faces or edges. For later use, we note that the integrals
\begin{align*}
  I(b) &:= \int_{K_+}\bigg(|b\na\phi| + C_\mu\sum_{i=1}^d|\na\phi_i|
	\bigg)^2 e^{b\phi_0 + C_\mu d}dx, \\
	J(b) &:= \sum_{j=1}^d\int_{f_j}\bigg(|b\na\phi| 
	+ C_\mu\sum_{i=1}^d|\na\phi_i|\bigg)^2 e^{b\phi_0 + C_\mu d}dx
\end{align*}
converge to zero as $b\to-\infty$, so there exists $L'_\mu\in\R$ such that
\begin{equation}\label{a.L}
  I(b) + J(b) \le 1 \quad\mbox{for all }b\le L'_\mu.
\end{equation}

We estimate
$$
  |\na\lambda_+\cdot n|\ge |a_0^\eps||\na\phi_0\cdot n|
	- \sum_{i=1}^d |a_i^\eps||\na\phi_0\cdot n|
	\ge |a_0^\eps||\na\phi_0\cdot n| - C_\mu d\max_{i=1,\ldots,d}|\na\phi_i|.
$$
As we assumed that  $p=1$, the expression $|\na\lambda_\eps\cdot n|$ is constant. 
Thus, since $|\na\phi_0\cdot n|>0$ and $\lambda_\eps\ge -C_\mu$ on $f_0$,
by \eqref{a.step3}, the previous inequality gives
\begin{align}
  |a_0^\eps| &\le \frac{1}{|\na\phi_0\cdot n|}\big(|\na\lambda_\eps\cdot n|
	+ C_\mu d\max_{i=1,\ldots,d}|\na\phi_i|\big) \nonumber \\
	&\le \frac{1}{|\na\phi_0\cdot n|}\bigg(\frac{e^{C_\mu}}{|f_0|}\bigg|\int_{f_0}
	e^{\lambda_\eps}\na\lambda_\eps\cdot n ds\bigg| 
	+ C_\mu d\max_{i=1,\ldots,d}|\na\phi_i|\bigg). \label{a.aux2}
\end{align}
An integration by parts leads to
\begin{align*}
  \int_{f_0}e^{\lambda_\eps}\na\lambda_\eps\cdot n ds
	+ \sum_{i=1}^d\int_{f_i}e^{\lambda_\eps}\na\lambda_\eps\cdot n ds
	&= \int_{K_+}e^{\lambda_\eps}\Delta\lambda_\eps dx 
	+ \int_{K_+}e^{\lambda_\eps}|\na\lambda_\eps|^2 dx \\
  &= \int_{K_+}e^{\lambda_\eps}|\na\lambda_\eps|^2 dx,
\end{align*}
since $\lambda_\eps$ is linear on $K_+$, so the Laplacian vanishes.
Hence, using 
$$
  \lambda_+ \le a_0^\eps\phi_0 + \sum_{i=1}^d |a_i^\eps||\phi_i|
	\le a_0^\eps\phi_0 + C_\mu d.
$$
and \eqref{a.L}, we have
\begin{align*}
  \bigg|\int_{f_0}e^{\lambda_\eps}\na\lambda_\eps\cdot n ds\bigg| 
	\le I(a_0^\eps) + J(a_0^\eps) \le 1
\end{align*}
if we choose $a_0^\eps\le L_\mu'$. Inserting this information into \eqref{a.aux2},
it follows for all $a_0^\eps\le L_\mu'$ that
$$
  |a_0^\eps| \le \frac{1}{|\na\phi_0\cdot n|}\bigg(\frac{e^{C_\mu}}{|f_0|}
	+ C_\mu d\max_{i=1,\ldots,d}|\na\phi_i|\bigg) =: -L_\mu''.
$$
Thus, setting $L_\mu=\min\{L_\mu',L_\mu''\}$, we conclude that $a_0^\eps\ge L_\mu$.

{\em Step 6:} Combining the previous steps, we infer that there exists a constant
$g(\mu)>0$ such that 
$$
  \|\lambda_\eps\|_{L^\infty(K_+)} \le g(\mu).
$$
This estimate means that if $\lambda_\eps$ is bounded in some element with 
constant $\mu$, then $\lambda_\eps$ is bounded in the neighboring elements
with constant $g(\mu)$. Now, take an arbitrary element $K\in\T_h$. Then there
exists a finite sequence $K^0,K^1,\ldots,K^m$ of elements with $K^0=K_\eps$
and $K^m=K$ such that $K^{j-1}$ and $K^j$ are neighboring elements. 
Repeating the arguments of Steps 3-5, the bound
$\|\lambda_\eps\|_{L^\infty(K^1)}\le \mu':=g(\mu)$ implies that
$\|\lambda_\eps\|_{L^\infty(K^2)}\le g(\mu')=g(g(\mu))$. Thus, by iteration,
$$
  \|\lambda_\eps\|_{L^\infty(K)} 
	\le \underbrace{(g\circ\cdots\circ g)}_{\text{$m$ times}}(\mu).
$$
The upper bound is independent of $\eps$ and holds for all elements $K\in\T_h$. 
Consequently, $(\lambda_\eps)$ is bounded in $L^\infty(\Omega)$.

We deduce that there exists a subsequence (not relabeled) such that 
$\lambda_\eps\to\lambda$ strongly in $L^\infty(\Omega)$, recalling that
$V_h$ is finite-dimensional. In fact, the convergence holds in any norm. 
Thus, we can pass to the limit
$\eps\to 0$ in \eqref{dg}, and the limit equation is the same as \eqref{dg}
with $\eps=0$.
\end{proof}
\end{appendix}


\end{document}